\documentclass[12pt]{amsart}
\usepackage{amsmath,amssymb,latexsym, amsthm, amscd, mathrsfs, stmaryrd} 
\usepackage[linktocpage=true]{hyperref}
\hypersetup{colorlinks,linkcolor=blue,urlcolor=cyan,citecolor=blue}

\usepackage[all]{xy}
\usepackage{hyperref}
\usepackage{color}
\newcommand{\nc}{\newcommand}
\nc{\browntext}[1]{\textcolor{brown}{#1}}
\nc{\greentext}[1]{\textcolor{green}{#1}}
\nc{\redtext}[1]{\textcolor{red}{#1}}
\nc{\bluetext}[1]{\textcolor{blue}{#1}}
\nc{\brown}[1]{\browntext{ #1}}
\nc{\green}[1]{\greentext{ #1}}
\nc{\red}[1]{\redtext{ #1}}
\nc{\blue}[1]{\bluetext{ #1}}
\nc{\zb}[1]{\redtext{From zb: #1}}

\newtheorem{theorem}{Theorem}  [section]

\newtheorem{lemma}[theorem]{Lemma}
\newtheorem{proposition}[theorem]{Proposition}

\newtheorem{definition}[theorem]{Definition}

\theoremstyle{remark}
\newtheorem{remark}[theorem]{Remark}

\numberwithin{equation}{section}

\newcommand{\bB}{{\mathbf B }}
\newcommand{\bDel}{\boldsymbol{\Delta}}
\newcommand{\bBKH}{\acute{\mathbf H}}

\newcommand{\BKH}{\acute{H}}

\def \dB{\Theta}

\def \cR{{\mathcal R}}

\def \gr{\text{gr}}

\def \tH{{\widehat{H}}}
\def \tTH{\widehat{\Theta}}

\newcommand{\Aut}{\operatorname{Aut}\nolimits}

\newcommand{\Sym}{\operatorname{Sym}\nolimits}

\def \y{{B}}

\def \bTH { \boldsymbol{\Theta}}
\def \bDel{ \boldsymbol{\Delta}}

\newcommand{\mbf}{\mathbf}

\newcommand{\de}{\delta}

\newcommand{\hgt}{\text{ht}}

\newcommand{\N}{\mathbb N}

\newcommand{\ov}{\overline}
\newcommand{\qbinom}[2]{\begin{bmatrix} #1\\#2 \end{bmatrix} }
\newcommand{\Q}{\mathbb Q}
\newcommand{\sll}{\mathfrak{sl}}
\newcommand{\T}{\texttt{\rm T}}
\renewcommand{\t}{\boldsymbol{\omega}}  
\newcommand{\U}{\mbf U}
\newcommand{\K}{\mathbb K}

\newcommand{\arxiv}[1]{\href{http://arxiv.org/abs/#1}{\tt arXiv:\nolinkurl{#1}}}

\newcommand{\Ui}{{\mbf U}^\imath}

\newcommand{\vs}{\varsigma}
\newcommand{\Y}{\bB}
\newcommand{\Z}{\mathbb Z}

\newcommand{\TT}{\texttt{\rm T}} 

\def \fg{\mathfrak{g}}

\def \I{\mathbb{I}}
\def \II{\mathbb{I}_0}

\newcommand{\tUi}{\widetilde{{\mathbf U}}^\imath}
\newcommand{\tUigr}{{\text{gr}}\tUi}

\newcommand{\tUiD}{{}^{\text{Dr}}\tUi}
\newcommand{\tUiDgr}{{\text{gr}}\tUiD}

\newcommand{\tU}{\widetilde{\mathbf U}}

\newcommand{\wt}{\text{wt}}

\allowdisplaybreaks

\def \K{\mathbb{K}}
\def \R{\mathbb{R}}
\def \SS{\mathbb{S}}

\def \bvs{{\boldsymbol{\varsigma}}}

\begin{document}

\title[ A Drinfeld type presentation of affine $\imath$quantum groups, I]
{A Drinfeld type presentation of affine $\imath$quantum groups I: split ADE type}

\author[Ming Lu]{Ming Lu}
\address{Department of Mathematics, Sichuan University, Chengdu 610064, P.R.China}
\email{luming@scu.edu.cn}

\author[Weiqiang Wang]{Weiqiang Wang}
\address{Department of Mathematics, University of Virginia, Charlottesville, VA 22904, USA}
\email{ww9c@virginia.edu}



\begin{abstract}
We establish a Drinfeld type new presentation for the $\imath$quantum groups arising from quantum symmetric pairs of split affine ADE type, which includes the $q$-Onsager algebra as the rank 1 case. This presentation takes a form which can be viewed as a deformation of half an affine quantum group in the Drinfeld presentation.
\end{abstract}

\maketitle
\setcounter{tocdepth}{1}
 \tableofcontents

\section{Introduction}

\subsection{Background}

Affine Lie algebras distinguish themselves among Kac-Moody Lie algebras, largely due to 2 distinct constructions: a Serre presentation (as definition) and a loop algebra realization. Such a dual nature of affine Lie algebras has led to numerous applications to string theory, modular forms, algebraic combinatorics, and so on.

Remarkably, these constructions repeat themselves at the quantum level for Drinfeld-Jimbo quantum groups: besides the Serre presentation (as definition) \cite{Dr87, Jim85}, an affine quantum group admits a current realization, also known as Drinfeld's new presentation \cite{Dr88}.
An explicit isomorphism between these 2 presentations with proof was supplied by Beck \cite{Be94} in the untwisted affine type. Damiani \cite{Da12, Da15} consolidated part of Beck's arguments and established interconnections among different relations in the current realization. 

Drinfeld's current presentation of affine quantum groups has played a fundamental role in numerous subsequent algebraic, geometric, and categorical developments. It has been instrumental in the active area of (finite-dimensional) representation theory of affine quantum groups $\U$, in algebraic or combinatorial approach, by numerous authors including Chari, Kashiwara,  and their collaborators and others; see the survey paper \cite{CH10} for partial references; there have been connections with cluster algebras and monoidal categorification  \cite{HL15}. A powerful geometric approach to representation theory of $\U$ was developed by Ginzburg, Vasserot, and Nakajima (see \cite{V98, Nak00}).
Moreover, the Drinfeld presentation arises categorically in Hall algebras of coherent sheaves over (weighted) projective lines as initiated by Kapranov; see \cite{Ka97, BKa01, Sch04, DJX12}.

According to the $\imath$program as outlined in \cite{BW18}, various algebraic, geometric, and categorical constructions of quantum groups should be generalizable to $\imath$quantum groups arising from quantum symmetric pairs. As an $\imath$-analogue of Drinfeld double quantum groups $\tU$, the universal $\imath$quantum groups $\tUi$ were introduced by the authors in \cite{LW19a} (also see \cite{LW20}) as they arise naturally from the $\imath$Hall algebra constructions of $\imath$quivers; moreover, there is a braid group action on $\tUi$ which is realized by reflection functors in $\imath$Hall algebras \cite{LW19b}. A central reduction of $\tUi$ recovers the $\imath$quantum groups $\Ui =\Ui_\bvs$ with parameters  $\bvs \in \Q(v)^{\times,\I}$ introduced earlier by G.~Letzter and generalized by Kolb \cite{Let99, Let02, Ko14}.

We view $\imath$quantum groups as a vast generalization of quantum groups, as quantum groups can be regarded as $\imath$quantum groups of diagonal type. The $\imath$quantum groups of finite type are classified by Satake diagrams (or the real forms of complex simple Lie algebras) of roughly 2 dozens types, and there is an even richer family of affine $\imath$quantum groups. 
  The split and more generally quasi-split $\imath$quantum groups make sense in Kac-Moody generality and form a distinguished class of $\imath$quantum groups.  For some recent progress in representation theory of $\imath$quantum groups, see \cite{Wat19}.

\subsection{Goal}

The goal of this paper is to initiate a current realization of affine $\imath$quantum groups. We shall restrict ourselves to the split affine ADE type in this paper, where most of the new features are already present; the Drinfeld type presentation for affine $\imath$quantum group of rank one (known as $q$-Onsager algebra) is already new.
We expect that the current realizations of affine $\imath$quantum groups will open up algebraic, categorical, and geometric developments as did Drinfeld's presentation for the affine quantum group $\U=\U(\widehat{\fg})$.

Throughout the paper, we shall work with the universal $\imath$quantum group $\tUi$, which contains central generators $\K_i$, for $i\in \I$.  The corresponding results for $\Ui$ (with arbitrary parameters) can be obtained from $\tUi$ readily by a central reduction which specializes $\K_i$ to scalars.

\subsection{Rank one} 

The split affine $\imath$quantum group $\Ui =\Ui_\bvs (\widehat{\sll}_2)$, with parameters $\bvs =\{\vs_0, \vs_1\} \in (\Q(v)^\times)^2$, is known as $q$-Onsager algebra in the literature. 
There has been some attempts with mixed success by Baseilhac and collaborators toward a Drinfeld type presentation for $\Ui$, cf. \cite{BS10, BK20} and references therein.

Making an ansatz with the constructions for affine quantum group $\U(\widehat{\sll}_2)$ by Damiani \cite{Da93}, Baseilhac and Kolb \cite{BK20} defined the $q$-root vectors in the $q$-Onsager algebra $\Ui$ (where both parameters were set to be equal: $\vs_0=\vs_1$, for technical reasons), and established a PBW basis for $\Ui$. Along the way, an affine braid group action on $\Ui$ (when $\vs_0=\vs_1$) is given. Various relations among the $q$-root vectors were computed, but clearly they do not resemble the relations in Drinfeld's current presentation for $\U(\widehat{\sll}_2)$.

In this paper, we first upgrade the main results of \cite{BK20} for $\Ui$ to $\tUi$, such as the constructions of $v$-root vectors (denoted by $B_{1,k}, \acute{\Theta}_m$, for $k\in \Z, m\in \Z_{\ge 1}$) and their relations, with the help of a braid group action coming from the $\imath$Hall algebra realization of $\tUi$. By a central reduction this in turn allows us to obtain the $v$-root vectors and PBW basis for $\Ui$ with 2 arbitrary parameters $\{\vs_0, \vs_1\}$, somewhat improving \cite{BK20}.

As a point of departure, we introduce new $v$-imaginary root vectors $\Theta_m$ and especially $H_m$, for $m\ge 1$, and work with generating functions. This allows us to greatly simplify some of the BK relations to be Drinfeld type relations. The normalization from $\acute{\Theta}_m$ to $\Theta_m$ in rank 1 is motivated by and in turn helps to facilitate the realization of $q$-Onsager algebra via $\imath$Hall algebra of the projective line in a subsequent work joint with Ruan \cite{LRW20}. We also obtain new Drinfeld type relations among real $v$-root vectors. In this way, we formulate a Drinfeld type presentation for $\tUi$ as an algebra isomorphism $\tUiD \cong \tUi$; see Definition~\ref{def:DrOnsa} and Theorem~\ref{thm:Dr1}. 

Via generating functions in a variable $z$,
\[
{\bTH} (z)= 1+(v-v^{-1}) \sum_{m\ge 1} \Theta_{m} z^m, \qquad 
\bB(z)=\sum_{r\in \Z} B_{1,r} z^r, \qquad
\bDel(z) = \sum_{k\in\Z}  C^k z^k, 
\]
where $C=\K_\de$, the Drinfeld presentation for the $q$-Onsager algebra $\tUi$ is reformulated as:
\begin{align}
 {\bTH} (z)  \bTH(w)
 &= {\bTH} (w)  \bTH(z),
 \label{gf1} \\
 {\bTH} (z)  \bB(w)
& =  \frac{(1 -v^{-2}zw^{-1}) (1 -v^{2} zw C)}{(1 -v^{2}zw^{-1})(1 -v^{-2}zw C)}
  \bB(w) {\bTH} (z),
  \\
 (v^2z-w) \Y(z) \Y(w) & +(v^{2}w-z) \Y(w) \Y(z)
 \label{gf3} \\
& =\frac{v^{-2} \K_1}{v-v^{-1}} \bDel(zw) \big( (v^2z-w)\bTH(w) +(v^2w-z)\bTH(z) \big).
\notag
\end{align}

\subsection{Higher rank} 

Now let $\tUi =\tUi_{\bvs} (\widehat{\fg})$ be the universal affine $\imath$quantum group of split ADE type, where $\fg$ is the Lie algebra of type ADE with root datum $\II$. (The corresponding $\imath$quntum groups $\Ui$ were introduced in  \cite{BB10} and referred to as generalized $q$-Onsager algebras; cf. \cite{Ko14}.) By definition, $\tUi$ is generated by $B_i, \K_i, \K_i^{-1}$, for $i\in \I =\II \cup \{0\}$, where $\K_i$ are central and the $B_i$ satisfy some inhomogenous Serre relations; see \eqref{eq:S2}--\eqref{eq:S3}. The algebra $\tUi$ can be realized via an $\imath$Hall algebra construction associated to $\imath$quivers \cite{LW19a, LW20}; moreover, there are automorphisms $\T_i$ ($i\in \I$) of $\tUi$ which are realized as reflection functors on $\imath$Hall algebras, giving rise to an affine braid group action on $\tUi$ \cite{LW21b}. It is worth pointing out that such a braid group action is not a restriction of the braid group action of Lusztig on quantum groups \cite{Lus90, Lus94}.

In a way similar to \cite{Be94}, with the help of the braid group action and the rank one construction above, we construct real $v$-root vectors $B_{i,k}$ and imaginary $v$-root vectors $\acute{\Theta}_{i,m}$ (or $\Theta_{i,m}$, $H_{i,m}$), for $i\in \II, k\in \Z, m\in \Z_{\ge 1}$.
Then the formulation of relations among these elements are used to define a new Drinfeld type algebra $\tUiD$; see Definition~\ref{def:iDR}. Our main result (see Theorem~\ref{thm:ADE}) asserts that $\tUiD \cong \tUi$, providing a new presentation for the affine $\imath$quantum group $\tUi$.

The defining relations~\eqref{iDR1}--\eqref{iDR5} for $\tUiD$ turn out to be strikingly neat and similar to Drinfeld's current relations for $\U(\widehat{\fg})$. Actually, one can regard half the Drinfeld's realization of $\U(\widehat{\fg})$  as the associated graded algebra with respect to a filtration of $\tUiD$ over its central subalgebra. In other words, the relations \eqref{iDR2}, \eqref{iDR3b} and \eqref{iDR5} for $\tUiD$ look like Drinfeld's relations for $\U(\widehat{\fg})$ plus lower terms (involving powers of $C=\K_\de$). Yet another view of the relations \eqref{iDR2}, \eqref{iDR3b} and \eqref{iDR5} for $\tUiD$ is that they exhibit a hybrid phenomenon mixing relations in the current negative half with relations between current positive and negative generators for $\U(\widehat{\fg})$.

One new relation which was not present in the rank one case is the Serre type relation \eqref{iDR5}. As the Serre relations among $B_i =B_{i,0}$ \eqref{eq:S2} are inhomogeneous, it is understandable that the general Serre type relations \eqref{iDR5} among $B_{i,k}$ for $\tUiD$ are much more challenging to formulate than its counterpart for $\U(\widehat{\fg})$, where the RHS of \eqref{iDR5} is simply set to 0 as in a standard Serre relation. What is perhaps surprising to us is that such a relation can be formulated concretely after all.

Damiani \cite{Da12} made a careful analysis of the relations in Drinfeld's current realization  of $\U(\widehat{\fg})$, and showed that they can be derived from a few distinguished relations. To that end, the triangular decomposition of $\U(\widehat{\fg})$ was very helpful.
In contrast, $\tUi$ does not admit a triangular decomposition. Because of this, the verifications of the new relations for $\tUi$, in particular \eqref{iDR3b}--\eqref{iDR5}, require a very different strategy from \cite{Be94, Da12}, though our overall plan is somewhat similar by showing that all the new relations for $\tUi$ can be derived from a few simpler ones. 

By verifying all the new relations in $\tUi$ we obtain a homomorphism $\Phi: \tUiD \rightarrow \tUi$, and it remains to show that $\Phi$ is an isomorphism. Our argument of the surjectivity of $\Phi$ is adapted from the proof of Damiani \cite[Theorem~12.11]{Da12}.
The injectivity of $\Phi$ follows by applying some filtered algebra argument to reduce to the corresponding injectivity in the affine quantum group setting.

We obtain several natural variants of the current presentation of $\tUi$, including one in the generating function form similar to \eqref{gf1}--\eqref{gf3} for $q$-Onsager algebra; see Theorem~\ref{thm:ADEgf}.
A surprising bonus of working with $\tUi$ (instead of $\Ui$) is that the canonical central element $C$ in affine quantum group naturally appear as $\K_\delta =\prod_{i\in \I} \K_i^{a_i}$ associated to the basic imaginary root $\de =\sum_{i\in \I} a_i \alpha_i$. This is especially clear in a symmetrized variant of the current presentation of $\tUi$ (see Definition~\ref{def:i-DR-ref} and Proposition~\ref{prop:symm}).
This phenomenon is even more remarkable, as the classical ($v\mapsto 1$) limit of $\tUi$ or $\Ui$ does not contain the canonical central element of the affine Lie algebra $\widehat \fg$; see \S\ref{subsec:classical}.

\subsection{Applications} 

While the $\imath$Hall algebras are not used in this paper in any explicit manner, they have played a fundamental role in guiding our work. Since we have realized the universal $\imath$quantum group in its Serre presentation via $\imath$Hall algebra of $\imath$quivers \cite{LW19a, LW20}, it is natural to expect that $\imath$Hall algebras of coherent sheaves over (weighted) projective lines should provide a realization of $\tUi$ in a new {\em current} presentation (keeping in mind the classic works \cite{Ka97, BKa01, Sch04, DJX12} relating (weighted) projective lines to the current realization of affine quantum groups). Indeed, such a realization of the Drinfeld type presentation of universal $q$-Onsager algebra is given in \cite{LRW20}; the higher rank cases will be treated in \cite{LR21}. 
Preliminary computations at earlier stages of \cite{LRW20, LR21} on $\imath$Hall algebras of (weighted) projective lines have been very helpful in pinning down and confirming some new relations for $\tUi$. 

It is our hope that this work can be of interest to people with diverse algebraic, categorical, geometric backgrounds. With this in mind, we have tried to make the presentation in this paper to be self-contained and come up with proofs independent of $\imath$Hall algebras. (As a result, a multiple of proofs for various relations in $\Ui$ are known to us.) 

In sequels to this paper, we shall further generalize this work to obtain Drinfeld type presentations for affine $\imath$quantum groups beyond split ADE type; the split BCFG type is developed by Weinan Zhang in \cite{Zh21}. The new current presentations should play a basic role in developing an algebraic approach toward the representation theory of affine $\imath$quantum groups (cf., e.g., \cite{CH10}), to which we hope to return elsewhere. They may also help to shed new lights on the open problem of finding Drinfeld type presentations for twisted Yangians and variants (cf. \cite{M07}). 

In yet another direction, this work makes it possible to develop the geometric realization of affine $\imath$quantum groups via equivariant K-theory \cite{SuW21}, building on the works of Y.~Li and collaborators \cite{BKLW, Li19} and generalizing earlier works of Vasserot and Nakajima \cite{V98, Nak00}.

\subsection{Organization}

In Section~\ref{sec:Onsager}, we define the $v$-root vectors and give a Drinfeld type presentation for the $q$-Onsager algebra $\tUi$.

In Section~\ref{sec:main}, we define the $v$-root vectors for $\tUi$ of split affine ADE type, and formulate a Drinfeld type presentation for $\tUi$ in the form of isomorphism $\Phi: \tUiD \rightarrow \tUi$. We show $\Phi$ is a bijection under the assumption that $\Phi$ is an algebra homomorphism, while postponing the verification of the new relations in $\tUi$ to Section~\ref{sec:relation1}. 
We verify the new relations in $\tUi$ one-by-one in Section~\ref{sec:relation1}, with the most challenging ones being \eqref{iDR5} and \eqref{iDR2} for $c_{ij}=-1$.

In Section~\ref{sec:variants}, we offer some variants of the new presentation for $\tUi$, one in generating function format (Theorem~\ref{thm:ADEgf}), and in a symmetrized form which resembles Drinfeld's realization for $\U$ better (Definition~\ref{def:i-DR-ref} and Proposition~\ref{prop:symm}), and yet another one in terms of different $v$-imaginary root vectors (Theorem~\ref{thm:ADE1}). We also formulate a Drinfeld type presentation for $\Ui_\bvs$ (Theorem~\ref{thm:ADE2}).

\vspace{2mm}
\noindent {\bf Acknowledgement.}
We thank Shiquan Ruan for his collaboration on related Hall algebra projects, and thank Weinan Zhang who inspired us to simplify much our earlier proofs in \S\ref{subsec:iDR2=>iDR5}--\ref{subsec:iDR5=>iDR2}. ML thanks University of Virginia for hospitality and support. WW is partially supported by the NSF grant DMS-1702254 and DMS-2001351.

\section{The $q$-Onsager algebra and its Drinfeld type presentation}
  \label{sec:Onsager}

In this section, we derive Drinfeld type new relations among the generators of the universal $q$-Onsager algebra, and recast them in the form of generating functions. This is built on a reformulation and enhancement of the results in \cite{BK20} for $q$-Onsager algebra.
We introduce new imaginary root vectors in the universal $q$-Onsager algebra (with motivation coming from the $\imath$Hall algebra of the projective line), and establish a Drinfeld type presentation.

\subsection{Root vectors}

For $n\in \Z, r\in \N$, denote by
\[
[n] =\frac{v^n -v^{-n}}{v-v^{-1}},\qquad
\qbinom{n}{r} =\frac{[n][n-1]\ldots [n-r+1]}{[r]!}.
\]
For $A, B$ in a $\Q(v)$-algebra, we shall denote $[A,B]_{v^a} =AB -v^aBA$, and $[A,B] =AB - BA$.

\begin{definition}
  \label{def:Onsager}
The {\em (universal) $q$-Onsager algebra $\tUi =\tUi(\widehat{\mathfrak{sl}}_2)$} is the $\Q(v)$-algebra generated by $B_0,B_1$, $\K_0,\K_1$, subject to the following relations: $\K_0,\K_1$ are central, and
\begin{align}
\sum_{r=0}^3 (-1)^r \qbinom{3}{r} B_i^{3-r} B_j B_i^{r}&= -v^{-1} [2]^2 (B_iB_j-B_jB_i) \K_i, 
\quad \text{ for } i\neq j.
\label{relation:s3}
\end{align}
(This algebra is also known as the {\rm universal $\imath$quantum group of split type $A_1^{(1)}$}.)
\end{definition}

\begin{remark}  \label{rem:Ui}
The generator $\K_i$ here is related to $\widetilde{k}_i$ used in \cite{LW19a} by $\K_i=-v^2 \widetilde{k}_i$; see Remark~\ref{rem:Kk}. The $q$-Onsager algebra $\Ui_\bvs$, for $\bvs =(\vs_0, \vs_1) \in (\Q(v)^\times)^2$, is obtained from $\tUi$ by a central reduction $\Ui_\bvs = \tUi /( \K_i +v^2 \vs_i | i=0,1)$, where $(-)$ denotes an ideal. The 1-parameter specialization $\Ui_\bvs$ by taking $\vs_0=\vs_1=-c$  recovers the $q$-Onsager algebra $\mathcal B_c$ studied in \cite{BK20}.
\end{remark}

Let $\{\alpha_0, \alpha_1\}$ be the simple roots of the affine Lie algebra $\widehat{\mathfrak{sl}}_2$, and $\de =\alpha_0 +\alpha_1$ be the basic imaginary root. The root system for $\widehat{\mathfrak{sl}}_2$ is
 $\cR =\{\pm (\alpha_1 + k \delta), m\delta \mid k,m \in \Z, m\neq 0\}.$
For $\mu, \nu  \in \Z \alpha_0 \oplus \Z \alpha_1$ and $i=0,1$,  set
\begin{align}
\K_{\alpha_i} =\K_i,  \quad
\K_{\alpha_i}^{-1} =\K_i^{-1}, \quad
\K_\delta =\K_0 \K_1,
\quad \K_{\mu +\nu} =\K_{\mu} \K_{\nu}.
\end{align}

Let $\dag$ be the involution of the $\Q(v)$-algebra $\tUi$ such that
\begin{align}
\dag:B_0\leftrightarrow B_1, \qquad \K_0\leftrightarrow \K_1.
\end{align}
We have the following two automorphisms $\TT_0,\TT_1$, which has an interpretation in $\imath$Hall algebras, see \cite{LW19b, LW21b}; this shows some conceptual advantage of $\tUi$ over $\Ui_\bvs$ (compare \cite{BK20}, where the parameters $\vs_i$ are set to be equal):
\begin{align}
\TT_1 (\K_1) &=\K_1^{-1},\qquad \TT_1(\K_0)= \K_{\delta} \K_1,
\\
\TT_1(B_1)&=  \K_1^{-1} B_1,
\\
\TT_1(B_0)&=  [2]^{-1} \big(B_0B_1^{2} -v[2] B_1 B_0B_1 +v^2 B_1^{2} B_0 \big) + B_0\K_1,
\label{T1B0}
\\
\TT_1^{-1}(B_0)&=  [2]^{-1} \big( B_1^{2}B_0-v[2] B_1B_0B_1 +v^2 B_0B_1^{2} \big) +B_0\K_1.
 \label{T1B0-2}
\end{align}
%
%
The action of $\TT_0$ is obtained from the above formulas by switching indices $0,1$, that is,
\begin{align}
\TT_0=\dag \circ \TT_1 \circ \dag.
\end{align}

For $n\in\Z$, following \cite{BK20}, we define the real $v$-root vectors
\begin{align}
B_{1,n} &=(\dag \TT_1)^{-n}(B_1).
  \label{eq:B1n}
\end{align}
Slightly modifying \cite{BK20}, we further define, for $m\ge 1$,
\begin{align}
\acute{\Theta}_{m} &= -B_{1,m-1} B_0+v^{2} B_0B_{1,m-1} + (v^{2}-1)\sum_{p=0}^{m-2} B_{1,p} B_{1,m-p-2} \K_0.
  \label{eq:dB1}
\end{align}
Note that $B_{1,0}=B_1$ by definition. (Our $- v^{-2} \acute{\Theta}_{m}$ corresponds to $ B_{m\delta}$ in \cite[(3.11)]{BK20}.)
In particular, we have
$
\acute{\Theta}_1 = -B_1B_0+v^{2} B_0B_1.
$ 
We also set
\begin{align}
\acute{\Theta}_{m} :=\begin{cases}
\frac{1}{v-v^{-1}} & \text{ if }m=0,
\\
0& \text{ if }m<0.
\end{cases}
\end{align}

From \eqref{eq:B1n}, we have
\begin{align*}
B_{1,-1}=(\dag \TT_1)(B_1)=B_0 \K_0^{-1},
\qquad
B_0=B_{1,-1}\K_0.
\end{align*}
So \eqref{eq:dB1} can be rewritten as
\begin{align}
\label{eq:reformTheta}
\acute{\Theta}_{m} &= \big(-B_{1,m-1} B_{1,-1}+v^{2} B_{1,-1}B_{1,m-1} + (v^{2}-1)\sum_{p=0}^{m-2} B_{1,p} B_{1,m-p-2} \big)\K_0
\\\notag
&=-\sum_{p=0}^{m-1} [B_{1,p},B_{1,m-2-p}]_{v^2} \K_0.
\end{align}
We note that \cite[Corollary~ 5.12]{BK20} in our setting of $\tUi$ reads
\begin{align}   \label{eq:Tm}
\dag \TT_1 (\acute{\Theta}_{m})=\acute{\Theta}_{m}.
\end{align}
Applying $(\dag\TT_1)^{-1}$ to \eqref{eq:reformTheta},
we have
$
\acute{\Theta}_{m} = -\sum_{p=1}^{m} [B_{1,p},B_{1,m-p}]_{v^2} \K_1^{-1}.
$

\subsection{Relations \`a la Baseilhac-Kolb} 

The following relations in $\tUi$ are the counterparts of the main relations for $\mathcal B_c$ established in \cite{BK20}; see Remark~\ref{rem:Ui}. They are obtained by literally repeating the arguments {\em loc. cit.}, and hence we shall skip the proofs altogether in this subsection. Our formulation in turn strengthens somewhat the results in \cite{BK20}, as the central reduction in Remark~\ref{rem:Ui} provides relations and then a presentation for $\Ui$ with {\em arbitrary} 2 parameters $\vs_i$ $(i=0, 1)$.

\begin{proposition}
[\text{\cite[Corollary 5.11]{BK20}}]
 \label{prop:ThTh1}
We have $[\acute{\Theta}_{n},\acute{\Theta}_{m}]=0$ holds in $\tUi$, for $n,m\ge 1$.
\end{proposition}

For $m\in\N$, define
\begin{align}
\label{eq:coeff}
a_p^m:=\left\{ \begin{array}{ll} v^{2(p-1)}(1+v^{2}), & \text{ if }p=1,2,\dots,\lfloor \frac{m-1}{2}\rfloor,
\\
v^{m-2}, & \text{ if }2 | m \text{ and }  p=\frac{m}{2}.  \end{array} \right.
\end{align}

\begin{proposition}
[\text{\cite[Proposition 5.5, Corollary 5.13]{BK20}}]
The following relation holds in $\tUi$, for $m\in\N$ and $r\in\Z$:
\begin{align}
\label{eq:BK581}
[ B_{1,r+m+1},&B_{1,r}]_{v^{2}}= -\acute{\Theta}_{m+1} \K_{r\delta+\alpha_1}-(v^{2}-1) \sum_{p=1}^{\lfloor \frac{m}{2}\rfloor} v^{2(p-1)} \acute{\Theta}_{m-2p+1}\K_{(p+r)\delta+\alpha_1}
\\
&\qquad \qquad +(v^{2}-1) \sum_{p=1}^{\lfloor \frac{m+1}{2}\rfloor} a_p^{m+1} B_{1,r+p}B_{1,m+r-p+1}. \notag
\end{align}
\end{proposition}

\begin{proposition}
[\text{\cite[Proposition 5.8, Corollary 5.13]{BK20}}]
The following relation holds in $\tUi$, for $m\geq 1$ and $r\in\Z$:
\begin{align}
\label{eq:im-real2}
&[\acute{\Theta}_{m}, B_{1,r}]
\\
&= [2]  \Big(
v^{2(m-1)} B_{1,r+m}
 - (v^2 -v^{-2}) \sum_{h=1}^{m-1} v^{2(m-2h)} B_{1,r+m-2h} \K_{h\de}
 - v^{2(1-m)} B_{1,r-m} \K_{m \de}
 \Big)
 \notag
\\
& +  (v^2 -v^{-2}) \times
\notag
\\
& \sum_{a=1}^{m-1}
\Big(
  v^{2(a-1)} B_{1,r+a}
- (v^2 -v^{-2}) \sum_{h=1}^{a-1} v^{2(a-2h)} B_{1,r+a-2h}  \K_{h\de}
- v^{2(1-a)} B_{1,r-a} \K_{a \de}
 \Big)
\acute{\Theta}_{m-a}.
\notag
\end{align}
\end{proposition}

\subsection{Drinfeld type relations in rank 1}

We shall introduce a new imagnary $v$-root vectors $\BKH_m$ and
formulate several Drinfeld type relations in $\tUi$ among the $v$-root vectors.

\subsubsection{} We start with the relations among real $v$-root vectors $B_{1,k}$.

\begin{proposition}   \label{prop:iDr}
The following relation holds in $\tUi$, for $r, s \in \Z$:
\begin{align}
  \label{rel:iDr}
[B_{1,r}, B_{1,s+1}]&_{v^{-2}}  -v^{-2}[B_{1,r+1}, B_{1,s}]_{v^{2}} = v^{-2}\acute{\Theta}_{r-s+1}\K_{s\de+\alpha_1} -v^{-2}\acute{\Theta}_{r-s-1}\K_{(s+1)\de+\alpha_1} \\
&+v^{-2}\acute{\Theta}_{s-r+1}\K_{r\de+\alpha_1} -v^{-2}\acute{\Theta}_{s-r-1}\K_{(r+1)\de+\alpha_1}.\notag
\end{align}
\end{proposition}

\begin{proof}
Note that $-v^{-2}[B_{1,r+1}, B_{1,s}]_{v^{2}} = [B_{1,s}, B_{1,r+1}]_{v^{-2}}$, and hence \eqref{rel:iDr} is invariant under $r\leftrightarrow s$. So we can assume  $m=s-r\geq0$.

By combining 2 relations in \eqref{eq:BK581} for $[B_{1,r+m+1},B_{1,r}]_{v^2}$ and $[B_{1,r+m},B_{1,r+1}]_{v^2}$ for $m\ge 2$, respectively, we obtain
\begin{align}
\label{rel:iDr1}
&[B_{1,r+m+1},B_{1,r}]_{v^2} - v^2 [B_{1,r+m},B_{1,r+1}]_{v^2}
\\
=&-\acute{\Theta}_{m+1} \K_{r\delta+\alpha_1} + \acute{\Theta}_{m-1}\K_{(r+1)\delta+\alpha_1} +(v^4-1) B_{1,r+1} B_{1,m+r}.\notag
\end{align}
One then rewrites \eqref{rel:iDr1} equivalently as \eqref{rel:iDr}, for $m\ge 2$.

Recalling $\acute{\Theta}_{0}=\frac{1}{v-v^{-1}}$, one observes that \eqref{rel:iDr} for $m=s-r=1$ is equivalent to the relation
$
[B_{1,r}, B_{1,r+2}]_{v^{-2}} = v^{-2}\acute{\Theta}_{2} \K_{r\delta+\alpha_1} + (v^{-2} -1) B_{1,r+1}^2,$
and then equivalent to the following relation in \eqref{eq:BK581}:
\[
[B_{1,r+2},B_{1,r}]_{v^{2}} = -\acute{\Theta}_{2} \K_{r\delta+\alpha_1} + (v^2 -1) B_{1,r+1}^2.
\]

The relation \eqref{rel:iDr} for $m=s-r=0$ is equivalent to $[B_{1,r}, B_{1,r+1}]_{v^{-2}} = v^{-2}\acute{\Theta}_{1}\K_{r\delta+\alpha_1}$, another relation in \eqref{eq:BK581}.
The proof of  \eqref{rel:iDr} is completed.
\end{proof}

\begin{remark}
The relation \eqref{rel:iDr} was derived from \eqref{eq:BK581} above; they are actually equivalent, as the converse can be shown by induction on $m$.
\end{remark}

\subsubsection{}

Define elements $\BKH_m$ in $\tUi$, for $m\ge 1$, by the following equation:
\begin{align}
\label{eq:exp}
1+ \sum_{m\geq 1} (v-v^{-1})\acute{\Theta}_{m} z^m  = \exp\Big( (v-v^{-1}) \sum_{m\ge 1}  \BKH_m z^m \Big).
\end{align}
Introduce the following generating functions in a variable $z$:
\[
\bBKH(z) =\sum_{m\ge 1} \BKH_m z^m,
\quad
\bB(z) =\sum_{r\in \Z} B_{1,r} z^r,
\quad
\acute\bTH (z) = 1+ (v-v^{-1})\sum_{m\ge 1}  \acute{\Theta}_{m} z^m.
\]
Then we have
\begin{align}  \label{Theta:z}
\acute\bTH (z) = \exp\big( (v-v^{-1}) \bBKH (z) \big).
\end{align}

The relations between imaginary and real $v$-root vectors can now be formulated as follows; here the imaginary $v$-root vectors refer to $\BKH_m$.

\begin{proposition}
   \label{prop:equiv4}
The following identities hold, for $m \ge 1, l\in \Z$:
\begin{align}
\label{eq:hB}
[\BKH_m, B_{1,l}] &=\frac{[2m]}{m} B_{1,l+m}-\frac{[2m]}{m} B_{1,l-m}\K_{m\delta},
\\
 \label{eq:eHBe}
\acute\bTH (z)  \bB(w)
& =   \frac{(1 -v^{-2}zw^{-1}) (1 -v^{2} zw \K_\de)}{(1 -v^{2}zw^{-1})(1 -v^{-2}zw \K_\de)}
  \bB(w) \acute\bTH (z), 
\\
\label{eq:hB1}
[\acute{\Theta}_{m},B_{1,l}]+[\acute{\Theta}_{m-2},B_{1,l}]\K_\de
& =v^{2}[\acute{\Theta}_{m-1},B_{1,l+1}]_{v^{-4}}+v^{-2}[\acute{\Theta}_{m-1},B_{1,l-1}]_{v^{4}}\K_\de.
\end{align}
Indeed, the identities \eqref{eq:im-real2}, \eqref{eq:hB}, \eqref{eq:eHBe} and \eqref{eq:hB1} are all equivalent.
\end{proposition}
The proof of Proposition~\ref{prop:equiv4} will be given in \S\ref{subsec:proof} below.

\subsubsection{}

One also has the following variant \eqref{eq:im-real} of \eqref{eq:im-real2}, which we will not use. We skip a similar proof.  
\begin{align}
\label{eq:im-real}
&[\acute{\Theta}_{m}, B_{1,l}]
\\
& = [2]   \Big(
v^{2(1-m)} B_{1,l+m}
+ (v^2 -v^{-2}) \sum_{h=1}^{m-1} v^{2(2h-m)} B_{1,l+m-2h}  \K_{h\de}
- v^{2(m-1)} B_{1,l-m} \K_{m \de}
 \Big)
 \notag
\\
& +  (v^2 -v^{-2}) \times
\notag
\\
& \sum_{a=1}^{m-1}  \acute{\Theta}_{m-a}
 \Big(
 v^{2(1-a)} B_{1,l+a}
 + (v^2 -v^{-2}) \sum_{h=1}^{a-1} v^{2(2h-a)} B_{1,l+a-2h}  \K_{h\de}
 - v^{2(a-1)} B_{1,l-a}  \K_{a \de}
 \Big). \notag
\end{align}

\subsection{Proof of Proposition~\ref{prop:equiv4} }
 \label{subsec:proof}

We shall establish the equivalences among identities \eqref{eq:im-real2}, \eqref{eq:hB}, \eqref{eq:eHBe} and \eqref{eq:hB1}.
\subsubsection{Proof of equivalences of \eqref{eq:im-real2}, \eqref{eq:hB} and \eqref{eq:eHBe} }

The identity \eqref{eq:hB} can be equivalently reformulated via generating functions as
\begin{align*}
(v-v^{-1}) [\bBKH(z), \bB(w)]
& =\sum_{k\ge 1,m\in\Z} B_{1,m+k} w^{m+k} \left( \frac{(v^{2} z w^{-1})^k}{k} - \frac{(v^{-2} z w^{-1})^k}{k}
\right)
\\
&\quad -\sum_{k\ge 1,m\in\Z} B_{1,m-k} w^{m-k} \left( \frac{ (v^{2} zw\K_\de )^k}{k}  - \frac{ (v^{-2} zw\K_\de )^k}{k} \right)
\\
&=  \ln \left(
  \frac{1-v^{-2}zw^{-1}}{1-v^{2}zw^{-1}} \cdot \frac{1-v^{2} zw \K_\de}{1-v^{-2}zw \K_\de}
  \right)
  \bB(w).
\end{align*}
Via integration this is then equivalent to
\begin{align}  \label{eq:eHBe1}
e^{(v-v^{-1}) \bBKH(z)} \bB(w) e^{ - (v-v^{-1}) \bBKH(z)}
& = \bB(w)
  \frac{(1 -v^{-2}zw^{-1}) (1 -v^{2} zw \K_\de)}{(1 -v^{2}zw^{-1})(1 -v^{-2}zw \K_\de)},
\end{align}
which can be reformulated as \eqref{eq:eHBe}.

We have the following identities:
\begin{align*}
\frac{1 -v^{-2}zw^{-1}}{1 -v^{2}zw^{-1}}
&= 1 +(v^2 -v^{-2}) \sum_{h\ge 1} v^{2(h-1)} z^h w^{-h},
\\
\frac{1 -v^{2} zw \K_\de}{1 -v^{-2}zw \K_\de}
&= 1 - (v^2 -v^{-2}) \sum_{h\ge 1} v^{2(1-h)} z^h w^{h} \K_{h\de},
\end{align*}
and hence
\begin{align}
&\frac{(1 -v^{-2}zw^{-1}) (1 -v^{2} zw \K_\de)}{(1 -v^{2}zw^{-1})(1 -v^{-2}zw \K_\de)}
 \label{eq:zwzw} \\
&= 1 + \sum_{a \ge 1} z^a (v^2 -v^{-2})
 \Big(
  v^{2(a-1)} w^{-a}
 - (v^2 -v^{-2}) \sum_{h=1}^{a-1} v^{2(a-2h)} w^{2h-a} \K_{h\de} -v^{2(1-a)} w^a \K_{a \de}
 \Big).
 \notag
\end{align}

It follows by \eqref{eq:zwzw} that the identity \eqref{eq:eHBe} is equivalent to the following identity:
\begin{align*}
& \Big(1+ \sum_{m \geq 1} (v-v^{-1})\acute{\Theta}_{m } z^m \Big) \bB(w)
=
\bB(w) \Big(1+ \sum_{m \geq 1} (v-v^{-1})\acute{\Theta}_{m} z^m \Big) +
\\
& \quad  \sum_{a \ge 1} z^a (v^2 -v^{-2})
 \left(
  v^{2(a-1)} w^{-a}
  - (v^2 -v^{-2}) \sum_{h=1}^{a-1} v^{2(a-2h)} w^{2h-a} \K_{h\de}
  -v^{2(1-a)} w^a \K_{a \de}
 \right) \times
 \\
 &\qquad\qquad
\bB(w) \Big(1+ \sum_{n \geq 1} (v-v^{-1})\acute{\Theta}_{n} z^n \Big).
\end{align*}
Equating the coefficients of $z^m w^l$ on both sides, for $m\ge 1$, we obtain
\begin{align*}
& (v-v^{-1}) [\acute{\Theta}_{m}, B_{1,l} ]
\\
& = (v^2 -v^{-2}) \times
\\
& \Big(
v^{2(m-1)} B_{1,l+m}
 - (v^2 -v^{-2}) \sum_{h=1}^{m-1} v^{2(m-2h)} B_{1,l+m-2h} \K_{h\de}
 - v^{2(1-m)} B_{1,l-m} \K_{m \de}
 \Big)
\\
& +  (v -v^{-1}) (v^2 -v^{-2}) \times
\\
& \sum_{a=1}^{m-1}
\Big(
  v^{2(a-1)} B_{1,l+a}
- (v^2 -v^{-2}) \sum_{h=1}^{a-1} v^{2(a-2h)} B_{1,l+a-2h}  \K_{h\de}
- v^{2(1-a)} B_{1,l-a} \K_{a \de}
 \Big)
\acute{\Theta}_{m-a},
\end{align*}
which is equivalent to the identity \eqref{eq:im-real2}.


\subsubsection{Equivalence of \eqref{eq:eHBe} and \eqref{eq:hB1} }

The identity \eqref{eq:eHBe} can be rephrased as
\begin{align*}
(1-v^2zw^{-1})(1-v^{-2}zw\K_\de) \acute{\bTH}(z) \bB(w)
=(1-v^{-2}zw^{-1})(1-v^2zw\K_\de)\bB(w) \acute{\bTH}(z).
\end{align*}
This can be rewritten as
\begin{align}  \label{eq:TBBT}
&(1-v^2zw^{-1})(1-v^{-2}zw\K_\de)\big(1+\sum_{m\geq1} (v-v^{-1})\acute{\Theta}_{m}z^m \big) \sum_{r\in\Z}B_{1,r}w^r
\\
=&(1-v^{-2}zw^{-1})(1-v^2zw\K_\de) \sum_{r\in\Z}B_{1,r}w^r\big(1+\sum_{m\geq1} (v-v^{-1})\acute{\Theta}_{m}z^m \big).
\notag
\end{align}
Equating the coefficients of $z^m w^l$ on both sides of \eqref{eq:TBBT}, for $m\ge 1$, we obtain the following.

If $m\geq3$, then
\begin{align*}
 \acute{\Theta}_m & B_{1,l}-v^2\acute{\Theta}_{m-1}B_{1,l+1} -v^{-2}\acute{\Theta}_{m-1}B_{1,l-1}\K_\de+\acute{\Theta}_{m-2}B_{1,l}\K_\de
\\
=&B_{1,l}\acute{\Theta}_m-v^{-2}B_{1,l+1}\acute{\Theta}_{m-1}-v^2B_{1,l-1}\acute{\Theta}_{m-1}\K_\de+B_{1,l}\acute{\Theta}_{m-2}\K_\de,
\end{align*}
which can be transformed into \eqref{eq:hB1}.

If $m=2$, then
\begin{align*}
&\acute{\Theta}_2B_{1,l}-v^2\acute{\Theta}_1 B_{1,l+1} -v^{-2}\acute{\Theta}_1 B_{1,l-1}\K_\de
=B_{1,l}\acute{\Theta}_2-v^{-2}B_{1,l+1}\acute{\Theta}_1-v^{2}B_{1,l-1}\acute{\Theta}_1\K_\de.
\end{align*}
If $m=1$, then
\begin{align*}
(v-v^{-1}) & \acute{\Theta}_1B_{1,l}-v^2B_{1,l+1} -v^{-2} B_{1,l-1}\K_\de
\\
=&(v-v^{-1})B_{1,l}\acute{\Theta}_1-v^{-2}B_{1,l+1}-v^{2}B_{1,l-1}\K_\de.
\end{align*}
The above identities in both cases for $m=1,2$ coincide with \eqref{eq:hB1}. Hence, the equivalence between  \eqref{eq:eHBe} and \eqref{eq:hB1} in all cases are established.


%
%
\subsection{New imaginary root vectors}

With application to $\imath$Hall algebra of the projective line \cite{LRW20} in mind, we shall introduce a somewhat modified versions of the elements $\acute{\Theta}_{m}$ and $\BKH_m$, denoted by $\Theta_m$ and $H_m$, respectively. In the $\imath$Hall algebra realization of the $q$-Onsager algebra \cite{LRW20}, $\Theta_m$ ($m\ge 1$) admit a natural interpretation in terms of torsion sheaves on the projective line. 

For $m \geq 1$, define
\begin{align}
\label{eq:dBB}
\dB_{m}=\acute{\Theta}_{m} - \sum\limits_{a=1}^{\lfloor\frac{m-1}{2}\rfloor}(v^2-1) v^{-2a} \acute{\Theta}_{m-2a}\K_{a\de} -\de_{m,ev}v^{1-m} \K_{\frac{m}{2}\de},
\end{align}
where
\begin{align*}
\de_{m,ev}= \begin{cases}
1, & \text{ for $m$ even}, \\
0, & \text{ for $m$ odd}.
\end{cases}
\end{align*}
Note that $\dB_{1}=\acute{\Theta}_{1}$.
We also set $\dB_{0}=\frac{1}{v-v^{-1}}$, and $\dB_{m}=0$ for $m<0$.
Let
\begin{align}   \label{Theta:z2}
\bTH (z) = 1+(v-v^{-1}) \sum_{m\ge 1} \Theta_{m} z^m.
\end{align}
We define the new {\em imaginary $v$-root vectors} $H_m$ by letting
\begin{align}  \label{eq:HTh}
1+ \sum_{m\geq 1} (v-v^{-1})\dB_{m} u^m  =  \exp\Big( (v-v^{-1}) \sum_{m\ge 1} H_m u^m \Big).
\end{align}

\begin{lemma}
The identity \eqref{eq:dBB} can be reformulated as a generating function identity:
\begin{align}
\label{eq:Theta2}
\bTH(z) =  \frac{1- \K_\delta z^2}{1-v^{-2}\K_\delta  z^2}  \acute{\bTH} (z).
\end{align}
\end{lemma}

\begin{proof}
By \eqref{eq:dBB}, we have (with a change of variables $k=m-2a$ in the double summand below)
\begin{align*}
\bTH (z) &=
1 + (v-v^{-1}) \sum_{m\ge 1} \acute{\Theta}_{m} z^m
- (v-v^{-1}) \sum_{m\ge 1} \sum_{m/2> a\ge 1} (v^2 -1) v^{-2a} \acute{\Theta}_{m-2a} \K_{a \delta} z^m
\\
&\quad - (v-v^{-1}) \sum_{n\ge 1} v^{1-2n} \K_{n \delta} z^{2n}
\\
&= \acute{\bTH}(z) - (v^2-1) \Big(\sum_{a\ge 1} v^{-2a} \K_{a \delta} z^{2a} \Big) \cdot \Big( \sum_{k\ge 1} (v -v^{-1}) \acute{\Theta}_{k} z^{k}\Big)
- \frac{(v-v^{-1}) v^{-1}\K_\delta z^2}{1-v^{-2}\K_\delta z^2}
\\
&= \acute{\bTH}(z) - (v^2-1)  \frac{v^{-2}\K_\delta z^2}{1-v^{-2}\K_\delta z^2} \big(\acute{\bTH}(z) -1 \big)
 - \frac{(v-v^{-1}) v^{-1}\K_\delta z^2}{1-v^{-2}\K_\delta z^2}
\\
&= \frac{1 -\K_\delta z^2}{1-v^{-2}\K_\delta z^2} \acute{\bTH}(z).
\end{align*}

The lemma is proved.
\end{proof}


\begin{lemma}
We have, for $m\in\Z$,
\begin{align}
\label{eqn:real1}
\dB_{m+1}-v^{-2}\dB_{m-1} \K_{\de}=\acute{\Theta}_{m+1}-\acute{\Theta}_{m-1} \K_{\de}.
\end{align}
\end{lemma}

\begin{proof}
By \eqref{eq:Theta2}, we have $(1-v^{-2}\K_\delta z^2) \bTH(z) =  (1- \K_\delta z^2)  \acute{\bTH} (z).$ The lemma follows by comparing the coefficients of $z^{m+1}$ on both sides of this identity.
%
%
%
\end{proof}

\begin{proposition}
\label{prop:HH1}
We have $[H_m, H_n] =0 = [\dB_m, \dB_n]$, for $m,n\ge 1$.
\end{proposition}

\begin{proof}
Follows from Proposition~\ref{prop:ThTh1} by using \eqref{eq:Theta2} and noting $\K_\de$ is central. 
\end{proof}

The following is a counterpart via $H_m$ and $\dB_m$ of the identities \eqref{eq:hB} and  \eqref{eq:hB1}, and they look formally the same.
\begin{proposition}
\label{prop:HB1}
The following identity holds in $\tUi$, for $m\ge 1$ and $r\in\Z$:
\begin{align}
[H_m, B_{1,l}] &=\frac{[2m]}{m} B_{1,l+m}-\frac{[2m]}{m} B_{1,l-m}\K_{m\delta},
\label{eq:hB3}
\\
[\dB_{m},B_{1,r}]+[\dB_{m-2},B_{1,r}]\K_\de& =v^{2}[\dB_{m-1},B_{1,r+1}]_{v^{-4}}+v^{-2}[\dB_{m-1},B_{1,r-1}]_{v^{4}}\K_\de.
\label{eq:hB2}
\end{align}
\end{proposition}

\begin{proof}
By \eqref{eq:eHBe}, \eqref{eq:Theta2} and noting $\K_\de$ is central, we have
\begin{align*}
 {\bTH} (z)  \bB(w)
& =
  \frac{(1 -v^{-2}zw^{-1}) (1 -v^{2} zw \K_\de)}{(1 -v^{2}zw^{-1})(1 -v^{-2}zw \K_\de)}
  \bB(w) {\bTH} (z),
\end{align*}
which takes the same form as \eqref{eq:eHBe}. Now \eqref{eq:hB2} (which takes the same form as \eqref{eq:hB1}) follows by exactly the same argument for the equivalence between \eqref{eq:eHBe} and \eqref{eq:hB1} in \S\ref{subsec:proof}. Then \eqref{eq:hB3} follows from \eqref{eq:hB2} exactly as in the proof of Proposition~\ref{prop:equiv4}.
\end{proof}

\begin{proposition}
  \label{prop:BB1}
We have, for $r,s\in \Z$,
\begin{align}
\label{rel:iDr1b}
[B_{1,r},& B_{1,s+1}]_{v^{-2}}  -v^{-2} [B_{1,r+1}, B_{1,s}]_{v^{2}} \\\notag
&= v^{-2}\dB_{(s-r+1)\de} \K_{r\de +\alpha_1}-v^{\red{-4}}\dB_{(s-r-1)\de} \K_{(r+1)\de+\alpha_1}
\\
&\quad +v^{-2}\dB_{(r-s+1)\de} \K_{s\de +\alpha_1}-v^{\red{-4}}\dB_{(r-s-1)\de} \K_{(s+1)\de+\alpha_1}.\notag
\end{align}
\end{proposition}

\begin{proof}
Follows from \eqref{rel:iDr} by using  \eqref{eqn:real1}. (We have used red color to indicate the q-powers in \eqref{rel:iDr1b} different from those in \eqref{rel:iDr}.)
\end{proof}

\subsection{Grading and filtration}

We shall relabel the $v$-root vectors by a set of roots $\{\alpha_1 + k \delta \mid k\in \Z \} \cup \{m \delta \mid m \ge 1\}$:
\begin{equation}  \label{eq:weight}
B_{\alpha_1+k\delta} =B_{1,k},
\quad B_{m\de} =H_m, \quad \text{ for } k\in \Z, m\ge 1.
\end{equation}
Denote by $\tUi_0$  the subalgebra of $\tUi$ generated by $\Theta_m$, for $m\ge 1$. By Definition~\ref{def:Onsager},  the algebra $\tUi$ is $\Z \I (\equiv \Z\alpha_0\oplus \Z \alpha_1)$-graded by weights, with
\begin{align}  \label{eq:grade1}
\wt (B_i) =\alpha_i, \quad \wt (\K_i) =2\alpha_i, \quad \text{ for } i \in \{0,1\}.
\end{align}
It follows that $\wt (B_{1,k}) =\alpha_1+k\delta, \wt (H_m) =m\delta$, whence the notation \eqref{eq:weight}.

For any fixed positive integer $\texttt{h}$ (with a standard choice being $\texttt h=1$), the algebra $\tUi$ is endowed with a filtered algebra structure $|\cdot |_{\texttt h}$ by setting
\begin{align}
  \label{eq:hfilt}
  |B_1|_{\texttt h} =1, \quad
  |B_0|_{\texttt h} =\texttt{h}, \quad
  |\K_i|_{\texttt h} =0, \quad
  \quad \text{ for } i\in \{0,1\}.
  \end{align}
Then there is an algebra isomorphism relating the associated graded $\gr \tUi$ to half a quantum affine $\mathfrak{sl}_2$, $\U^- =\langle F_0, F_1 \rangle$ in the setting of \cite{LW20}, which goes back to \cite{Let02, Ko14}:
\begin{align}   \label{eq:filter1}
\gr \tUi \cong \U^- \otimes \Q(v)[\K_1^{\pm 1},\K_\de^{\pm 1}],
\qquad \overline{B_i} \mapsto F_i \; (i=0,1).
\end{align}

\begin{proposition} [cf. \cite{BK20}]
  \label{prop:PBW1}
  The following holds in $\tUi:$
\begin{enumerate}
\item
The subalgebra $\tUi_0$ is a polynomial algebra in $\Theta_m$, for $m\ge 1$;  it is also a polynomial algebra in $H_m$, for $m\ge 1$.
\item
Fix any total order $<$ on the roots $\{\alpha_1 + k \delta \mid k\in \Z \} \cup \{m \delta \mid m \ge 1\}$. Then $\tUi$ admits a basis
\[
\big \{\K_1^{r} \K_\de^{s} B_{\gamma_1}^{a_1} B_{\gamma_2}^{a_2} \ldots B_{\gamma_N}^{a_N} \mid
r, s\in\Z, a_1, a_2, \ldots, a_N \in \N, N\in \N, \gamma_1<\gamma_2 <\ldots < \gamma_N \big \}.
\]
\end{enumerate}
\end{proposition}

\begin{proof}
Part (2) follows as a variant of \cite[Theorem 4.5]{BK20} on a PBW basis for $\Ui$, which is proved by using the filtration \eqref{eq:hfilt}--\eqref{eq:filter1} and comparing with the $v$-root vectors in $\U$ given in \cite{Da93}. (A mild difference is that $\tUi$ admits the central elements $\K_\beta$, for $\beta \in \cR$, and $B_{m\de}$ for $\Ui$ used {\em loc. cit.} is understood as a version of $\acute\Theta_m$.)

Clearly, the algebraic independence among $\{\acute\Theta_m\mid m\ge 1\}$ implies the algebraic independence of
$\{\Theta_m\mid m\ge 1\}$ as well as of $\{H_m\mid m\ge 1\}$.
Part (1) follows.
\end{proof}
\subsection{A Drinfeld type presentation in rank 1}
 \label{subsec:Dr1}

\begin{definition}
  \label{def:DrOnsa}
Let $\tUiD =\tUiD (\widehat{\sll}_2)$ be the $\Q(v)$-algebra  generated by $\K_1^{\pm1}$, $C^{\pm1}$, $H_{m}$ and $B_{1,r}$, where $m\geq1$, $r\in\Z$, subject to the following defining relations, for $m,n\geq1$ and $r,s\in \Z$:
\begin{align}
\K_1\K_1^{-1} =1, \quad C C^{-1} &=1, \quad \K_1, C \text{ are central},
 \label{iDRo0}
\\
[H_m,H_n] &=0, \label{iDRo1}
\\
[H_m, B_{1,r}] &=\frac{[2m]}{m} B_{1,r+m}-\frac{[2m]}{m} B_{1,r-m}C^m,
\label{iDRo2}
\\
\label{iDRo3}
[B_{1,r}, B_{1,s+1}]_{v^{-2}}  -v^{-2} [B_{1,r+1}, B_{1,s}]_{v^{2}}
&= v^{-2}\dB_{s-r+1} C^r \K_1-v^{-4}\dB_{s-r-1} C^{r+1} \K_1 \\ \notag
&\quad  +v^{-2}\dB_{r-s+1} C^s \K_1-v^{-4}\dB_{r-s-1} C^{s+1} \K_1, \notag
\end{align}
where
$1+ \sum_{m\geq 1} (v-v^{-1})\dB_{m} z^m  =  \exp\big( (v-v^{-1}) \sum_{m\ge 1} H_m z^m \big).$
\end{definition}

\begin{theorem}
\label{thm:Dr1}
There is a $\Q(v)$-algebra isomorphism ${\Phi}: \tUiD \rightarrow\tUi$, which sends
\begin{align*}
B_{1,r}\mapsto B_{1,r}, \quad \Theta_m \mapsto \Theta_m,
\quad
\K_1\mapsto \K_1, \quad C\mapsto \K_\de,
\quad \text{ for } m\ge 1, r\in \Z.
\end{align*}
The inverse ${\Phi}^{-1} : \tUi \rightarrow \tUiD$ sends
$\K_1\mapsto  \K_1,
\K_0\mapsto  C \K_1^{-1},
B_1\mapsto   B_{1,0},
B_0\mapsto B_{1,-1}C \K_1^{-1}.$
\end{theorem}

\begin{proof}
The relations \eqref{iDRo1}, \eqref{iDRo2} and \eqref{iDRo3}  in ${}^{\text{Dr}}\tU$ hold for the images  of the generators of $\tUiD$ under $\Phi$, thanks to Propositions~\ref{prop:HH1}, \ref{prop:HB1}, and \ref{prop:BB1}, respectively. (The relation \eqref{iDRo0} under $\Phi$ holds trivially.) Thus $\Phi$ is a homomorphism. 

By the defining relations of $\tUiD$, one shows that $\tUiD$ has a spanning set
\[
\{\K_1^{r} C^{s} B_{\gamma_1}^{a_1} B_{\gamma_2}^{a_2} \ldots B_{\gamma_N}^{a_N} \mid
r, s\in\Z, a_1, a_2, \ldots, a_N \in \N, N\in \N, \gamma_1<\gamma_2 <\ldots < \gamma_N \}.
\]
Since this set is mapped to a basis of $\tUi$ according to Proposition~\ref{prop:PBW1}, it must be a basis for $\tUiD$ as well. Therefore, $\Phi$ is an isomorphism.
\end{proof}



\begin{remark}
There is another Drinfeld type presentation of the $\Q(v)$-algebra $\tUi$ with the same set of generators as in Theorem~\ref{thm:Dr1}, subject to the relations \eqref{iDRo0}--\eqref{iDRo2} and \eqref{iDRo3b} below (in place of \eqref{iDRo3}):
\begin{align}
\label{iDRo3b}
[B_{1,r}, B_{1,s+1}]_{v^{-2}}  -v^{-2} [B_{1,r+1}, B_{1,s}]_{v^{2}}
&= v^{-2}\dB_{s-r+1} C^r \K_1-v^{-2}\dB_{s-r-1} C^{r+1} \K_1 \\ \notag
&\quad  +v^{-2}\dB_{r-s+1} C^s \K_1-v^{-2}\dB_{r-s-1} C^{s+1} \K_1. \notag
\end{align}
One should (secretly) regard the $\dB_m$ in this presentation as $\acute{\Theta}_m$ defined earlier, and then the relation \eqref{iDRo3b} is simply \eqref{rel:iDr}. In this way, the equivalence between \eqref{rel:iDr} and \eqref{rel:iDr1b} (or \eqref{iDRo3}) has been established earlier. Hence the equivalence between this presentation and the presentation $\tUiD$ in Definition~\ref{def:DrOnsa} follows.
\end{remark}

\section{A Drinfeld type presentation of affine $\imath$quantum groups }
  \label{sec:main}

In this section, we study in depth the $\imath$quantum groups $\tUi$ of split affine ADE type. We introduce a new set of generators for $\tUi$ and use them to formulate a Drinfeld type presentation for $\tUi$, generalizing the one for $q$-Onsager algebra in \S\ref{subsec:Dr1}.

\subsection{Affine Weyl and braid groups}

Let $(c_{ij})_{i,j\in \II}$ be the Cartan matrix of the simple Lie algebra $\fg$ of type ADE. Let $\cR_0$ be the set of roots for $\fg$, and fix a set $\cR^+_0$ of positive roots  with simple roots $\alpha_i$ $(i\in \II)$. Denote by $\theta$ the highest root of $\fg$.

Let $\widehat{\fg}$ be the (untwisted) affine Lie algebra with affine Cartan matrix denoted by $(c_{ij})_{i,j\in\I}$, where $\I=\{0\} \cup \II$ with the affine node $0$. Let $\alpha_i$ $(i\in \I)$ be the simple roots of $\widehat{\fg}$, and $\alpha_0=\de -\theta$, where $\de$ denotes the basic imaginary root. The root system $\cR$  for $\widehat{\fg}$ and its positive system $\cR^+$ are defined to be
\begin{align}
\cR &=\{\pm (\beta + k \delta) \mid \beta \in \cR_0^+, k  \in \Z\}  \cup \{m \delta \mid m \in \Z\backslash \{0\} \},
     \label{eq:roots}  \\
\cR^+ &= \{k \delta +\beta \mid \beta \in \cR_0^+, k  \ge 0\}
\cup  \{k \delta -\beta \mid \beta \in \cR_0^+, k > 0\}
\cup \{m \delta \mid m \ge 1\}.
  \label{eq:roots+}
 \end{align}
For $\gamma =\sum_{i\in \I} n_i \alpha_i \in \N \I$, the height $\text{ht} (\gamma)$ is defined as $\text{ht} (\gamma) =\sum_{i\in \I} n_i$.

Let $P$ and $Q$ denote the weight and root lattices of $\fg$, respectively. Let $\omega_i \in P$ $(i\in \II)$ be the fundamental weights of $\fg$. Note $\alpha_i =\sum_{j\in \II} c_{ij}\omega_j$. We define a bilinear pairing $\langle \cdot, \cdot \rangle : P\times Q \rightarrow \Z$ such that $\langle \omega_i, \alpha_j \rangle =\delta_{i,j}$, for $i,j \in \II$, and thus $\langle \alpha_i, \alpha_j \rangle = c_{ij}$.


The Weyl group $W_0$ of $\fg$ is generated by the simple reflections $s_i$, for $i \in \II$. It acts on $P$ by
$s_i(x)=x-\langle x, \alpha_i \rangle\alpha_i$ for $x\in P$. The extended affine Weyl group $\widetilde{W}$ is the semi-direct product $W_0 \ltimes P$, which contains the affine Weyl group $W:=W_0 \ltimes Q =\langle s_i \mid i \in \I \rangle$ as a subgroup; we denote
\[
t_\omega =(1, \omega) \in \widetilde W, \quad \text{ for } \omega \in P.
\]
We identify $P/Q$ with a finite group $\Omega$ of Dynkin diagram automorphism, and so $\widetilde{W} =\Omega . W$. There is a length function $\ell(\cdot)$ on $\widetilde{W}$ such that $\ell(s_i)=1$, for $i\in \I$, and each element in $\Omega$ has length 0.

For $i\in \II$, as in \cite{Be94}, we have
\begin{equation}  \label{eq:tomega}
\ell(t_{\omega_i}) =\ell(\omega_i')+1, \qquad
\text{ where }  \omega_i':= t_{\omega_i} s_i \in W.
\end{equation}


Let $\U =\U(\widehat{\fg})$ denote the Drinfeld-Jimbo affine quantum group, a $\Q(v)$-algebra generated by $E_i, F_i, K_i^{\pm 1}$, for $i\in \I$.
 Following Lusztig \cite{Lus90, Lus94}, we have the braid group action of $T_w$ on $\tU$, for $w \in \widetilde W$; for example $T_i$, for $i\in \I$, acts on $\U$ by, for $i\neq j \in \I, \mu \in \Z\I$ (our $K_i$ corresponds to $\widetilde{K}_i$ in \cite{Lus94}):
\begin{equation}\label{eq:braidgroup}
\begin{split}
&T_i (E_i) = -F_i  {K}_{i} , \quad T_i (F_i) = - {K}_{i}^{-1}E_i,  \quad T_i (K_\mu) = K_{s_i(\mu)}, \\
&T_i (E_j) =  \sum_{r+s = - c_{ij}} (-1)^r v^{-r}_{i} E^{(s)}_i E_j E^{(r)}_i,  \\
&T_i (F_j) =  \sum_{r+s = - c_{ij}} (-1)^r v^{r}_{i} F^{(r)}_i F_j F^{(s)}_i.
\end{split}
\end{equation}
The formulas above are written in terms of $v_i$ (dependent on the length of $\alpha_i$) for the convenience of future references. In the current ADE setting, we always have $v_i =v$. Similar remarks apply elsewhere, e.g., $[n]_i =[n]_{v_i}$.

The following was known to Bernstein and Lusztig.
\begin{lemma} [\text{\cite[\S2.7]{Lus89}}]
   \label{lem:Braid-Lus}
Let $x\in P$, $i, j \in \II$.
\begin{enumerate}
\item[(a)]
If $s_ix=xs_i$, then $T_iT_x=T_xT_i$.
\item[(b)]
If $s_i xs_i=t_{\alpha_i}^{-1}x=\prod_{k\in \II} t_{\omega_k}^{a_k}$, then we have
$T_i^{-1}T_xT_i=\prod_k T_{\omega_k}^{a_k}$; in particular, we have $T_i^{-1}T_{\omega_i}T_i^{-1}=T_{\omega_i}^{-1}\prod_{k\neq i}T_{\omega_k}^{-c_{ik}}$.
\item[(c)]
$T_{\omega_i} T_{\omega_j}  = T_{\omega_j} T_{\omega_i}.$
\end{enumerate}
\end{lemma}

We record some useful formulas:
\begin{align}
T_i \big(T_{\omega_i'}(F_i) \big) 
&=T_{\omega_i'}(F_i) F_i^{(2)} -v_i F_i T_{\omega_i'}(F_i) F_i +v_i^2 F_i^{(2)} T_{\omega_i'}(F_i),
\label{eq:TiX}
\\
T_i^{-1} \big(T_{\omega_i'}(F_i) \big)
&= F_i^{(2)} T_{\omega_i'}(F_i) -v_i F_i T_{\omega_i'}(F_i) F_i +v_i^2 T_{\omega_i'}(F_i) F_i^{(2)},
\label{eq:TiX2}
\\
T_i \big([T_{\omega_i'}(F_i), F_i]_{v_i^{-2}} \big) &=[F_i, T_{\omega_i'}(F_i)]_{v_i^{-2}}.
  \label{T-imagine}
\end{align}
The formula \eqref{eq:TiX} follows by applying the Chevalley involution to its $E$-version, which can be found in \cite{Be94}. The formula \eqref{eq:TiX2} can be derived (and is equivalent to) \eqref{eq:TiX}. Formula \eqref{T-imagine} can also be derived directly from \eqref{eq:TiX}.

\subsection{Drinfeld presentation of an affine quantum group}

Let $C=(c_{ij})_{i,j\in\I}$ be the Cartan matrix of untwisted affine type. Let $^{\text{Dr}}\U$ be the $\Q(v)$-algebra generated by $x_{i k}^{\pm}$, $h_{i l}$, $K_i^{\pm 1}$, $C^{\pm \frac12}$, for $i\in\I$, $k\in\Z$, and $l\in\Z\backslash\{0\}$, subject to the following relations: $C^{\frac12}, C^{- \frac12}$ are central such that
\begin{align*}
[K_i,K_j] & =  [K_i,h_{j l}] =0, \quad K_i K_i^{-1} =C^{\frac12} C^{- \frac12} =1,
\\
[h_{ik},h_{jl}] &= \delta_{k, -l} \frac{[k c_{ij}]_i}{k} \frac{C^k -C^{-k}}{v_j -v_j^{-1}},
\\
K_ix_{jk}^{\pm} K_i^{-1} &=v^{\pm c_{ij}} x_{jk}^{\pm},
 \\
[h_{i k},x_{j l}^{\pm}] &=\pm\frac{[kc_{ij}]_i}{k} C^{\mp \frac{|k|}2} x_{j,k+l}^{\pm},
\\
[x_{i k}^+,x_{j l}^-] &=\delta_{ij} {(C^{\frac{k-l}2} K_i\psi_{i,k+l} - C^{\frac{l-k}2} K_i^{-1} \varphi_{i,k+l})}, 
 \\
x_{i,k+1}^{\pm} x_{j,l}^{\pm}-v^{\pm c_{ij}} x_{j,l}^{\pm} x_{i,k+1}^{\pm} &=v^{\pm c_{ij}} x_{i,k}^{\pm} x_{j,l+1}^{\pm}- x_{j,l+1}^{\pm} x_{i,k}^{\pm},
 \\
\Sym_{k_1,\dots,k_r}\sum_{t=0}^{r} (-1)^t \qbinom{r}{t}_i
& x_{i,k_1}^{\pm}\cdots
 x_{i,k_t}^{\pm} x_{j,l}^{\pm}  x_{i,k_t+1}^{\pm} \cdots x_{i,k_n}^{\pm} =0, \text{ for } r= 1-c_{ij}\; (i\neq j),
\end{align*}
where
$\Sym_{k_1,\dots,k_r}$ denotes the symmetrization with respect to the indices $k_1,\dots,k_r$, $\psi_{i,k}$ and $\varphi_{i,k}$ are defined by the following functional equations:
\begin{align*}
1+ \sum_{m\geq 1} (v_i-v_i^{-1})\psi_{i,m}u^m &=  \exp\Big((v_i -v_i^{-1}) \sum_{m\ge 1}  h_{i,m}u^m\Big),
\\
1+ \sum_{m\geq1 } (v_i-v_i^{-1}) \varphi_{i, -m}u^{-m} &= \exp \Big((v_i -v_i^{-1}) \sum_{m\ge 1} h_{i,-m}u^{-m}\Big).
\end{align*}
(We omit a degree operator $D$ in the version of ${}^{\text{Dr}}\U$ above.)

It was stated by Drinfeld \cite{Dr88} (and proved by Beck \cite{Be94} and Damiani \cite{Da12, Da15})
that there exists a $\Q(v)$-algebra isomorphism
\begin{align} \label{eq:phi1}
\phi: {}^{\text{Dr}}\U \longrightarrow \U.
\end{align}
We omit the explicit formulas for the images under $\phi$ of generators of ${}^{\text{Dr}}\U$; see \cite{Be94, BCP99}.

\subsection{Affine $\imath$quantum groups of split ADE type}

Recall the Cartan matrix $(c_{ij})_{i,j\in \I}$ of affine ADE type, for $\I = \II \cup \{0\}$ with the affine node $0$. 
The notion of (quasi-split) universal $\imath$quantum groups $\tUi$ was formulated in \cite{LW19a}.
The {\em universal affine $\imath$quantum group of split ADE type} is the $\Q(v)$-algebra $\tUi =\tUi(\widehat{\fg})$ with generators $B_i$, $\K_i^{\pm 1}$ $(i\in \I)$, subject to the following relations, for $i, j\in \I$:
\begin{align}
\K_i\K_i^{-1} =\K_i^{-1}\K_i=1, & \quad \K_i  \text{ is central},
\\
B_iB_j -B_j B_i&=0, \qquad\qquad\qquad\qquad\qquad \text{ if } c_{ij}=0,
 \label{eq:S1} \\
B_i^2 B_j -[2] B_i B_j B_i +B_j B_i^2 &= - v^{-1}  B_j \K_i,  \qquad\qquad\qquad \text{ if }c_{ij}=-1,
 \label{eq:S2} \\
\sum_{r=0}^3 (-1)^r \qbinom{3}{r} B_i^{3-r} B_j B_i^{r} &= -v^{-1} [2]^2 (B_iB_j-B_jB_i) \K_i,    \text{ if }c_{ij}=-2.
 \label{eq:S3}
 \end{align}
The universal split $\imath$quantum group of rank 1 is also known as the $q$-Onsager algebra, and this is the only case where the relation \eqref{eq:S3} is needed; see Definition~\ref{def:Onsager}.

For $\mu = \mu' +\mu''  \in \Z \I := \oplus_{i\in \I} \Z \alpha_i$,  define $\K_\mu$ such that
\begin{align}
\K_{\alpha_i} =\K_i, \quad
\K_{-\alpha_i} =\K_i^{-1}, \quad
\K_{\mu} =\K_{\mu'} \K_{\mu''},
\quad  \K_\delta =\K_0 \K_\theta.
\end{align}
The algebra $\tUi$ is endowed with a filtered algebra structure
\begin{align}  \label{eq:filt1}
\widetilde{\U}^{\imath,0} \subset \widetilde{\U}^{\imath,1} \subset \cdots \subset \widetilde{\U}^{\imath,m} \subset \cdots
\end{align}
by setting 
\begin{align}  \label{eq:filt}
\widetilde{\U}^{\imath,m} =\Q(v)\text{-span} \{ B_{i_1} B_{i_2} \ldots B_{i_r} \K_\mu \mid \mu \in \N\I, i_1, \ldots, i_r \in \I, r\le m \}.
\end{align}
Note that
\begin{align}  \label{eq:UiCartan}
\widetilde{\U}^{\imath,0} =\bigoplus_{\mu \in \N\I} \Q(v) \K_\mu,
\end{align}
is the $\Q(v)$-subalgebra generated by $\K_i$ for $i\in \I$.
Then, according to a basic result of Letzter and Kolb on quantum symmetric pairs adapted to our setting of $\tUi$ (cf. \cite{Let02, Ko14}), the associated graded $\gr \tUi$ with respect to \eqref{eq:filt1}--\eqref{eq:filt} can be identified with
\begin{align}   \label{eq:filter}
\gr \tUi \cong \U^- \otimes \Q(v)[\K_i^\pm | i\in \I],
\qquad \overline{B_i}\mapsto F_i,  \quad
\overline{\K}_i \mapsto \K_i \; (i\in \I).
\end{align}

\begin{remark}
 \label{rem:Kk}
The generator $\K_i$ here, which corresponds to a (generalized) simple module in the $\imath$Hall algebra, is related to the generator $\tilde{k}_i$ used in \cite{LW19a, LW20} (which is natural from the viewpoint of Drinfeld doubles) by a rescaling:
$\K_i = -v^2\tilde{k}_i$.
The precise relation between the algebra $\tUi$ and the $\imath$quantum group $\Ui$ arising from quantum symmetric pairs \cite{Ko14} is explained {\em loc. cit.}; also see \S\ref{subsec:parameter} below.
\end{remark}

\begin{remark}
The $\Q(v)$-algebra $\tUi$ is $\Z \I$-graded by letting
\begin{align}
 \label{eq:deg}
\wt (B_i) =\alpha_i, \quad \wt (\K_i) =2\alpha_i, \quad \text{ for } i \in \I.
\end{align}
A variant of $\tUi$, in which $\K_i$ is not assumed to be invertible, is $\N\I$-graded by \eqref{eq:deg}. \end{remark}

\begin{lemma} [\text{also cf. \cite{KP11, BK20}}]
\label{lem:Ti}
For $i\in I$, there exists an automorphism $\TT_i$ of the $\Q(v)$-algebra $\tUi$ such that
$\TT_i(\K_\mu) =\K_{s_i\mu}$, and
\[
\TT_i(B_j)= \begin{cases}
 \K_i^{-1} B_i,  &\text{ if }j=i,\\
B_j,  &\text{ if } c_{ij}=0, \\
B_jB_i-vB_iB_j,  & \text{ if }c_{ij}=-1, \\
 {[}2]^{-1} \big(B_jB_i^{2} -v[2] B_i B_jB_i +v^2 B_i^{2} B_j \big) + B_j\K_i,  & \text{ if }c_{ij}=-2,
\end{cases}
\]
for $\mu\in \Z\I$ and $j\in \I$.
Moreover,  $\TT_i$ $(i\in \I)$ satisfy the braid group relations, i.e., $\TT_i \TT_j =\TT_j \TT_i$ if $c_{ij}=0$, and $\TT_i \TT_j \TT_i =\TT_j \TT_i \TT_j$ if $c_{ij}=-1$.
\end{lemma}

\begin{proof}
The formulas for $\TT_i$ when $c_{ij}=0, -1$ were obtained in \cite{LW19b} via reflection functors in an $\imath$Hall algebra approach and the braid relations are verified therein; the formulas for $\TT_i$ when $c_{ij}= -2$ can be similarly obtained (as a special case of $\imath$quantum groups of Kac-Moody type in \cite{LW21b}); also see  \cite{KP11} and \cite{BK20} for closely related versions on $\Ui$ via computer packages. It follows from the reflection functor construction that these formulas define an automorphism $\TT_i$ of $\tUi$.
\end{proof}
We also have $\TT_i^{-1} (\K_\mu) =\K_{s_i\mu}$, and
\[
\TT_i^{-1} (B_j)=
\begin{cases}
 \K_i^{-1} B_i,  &\text{ if }j=i,\\
B_j,  &\text{ if } c_{ij}=0, \\
B_iB_j -vB_jB_i,  & \text{ if }c_{ij}=-1, \\
 {[}2]^{-1} \big( B_i^{2}B_j-v[2] B_iB_jB_i+v^2 B_jB_i^{2} \big) +B_j\K_i,  & \text{ if }c_{ij}=-2.
\end{cases}
\]
Just as the quantum group setting \cite{Lus94},  $\TT_i^{-1}$ is related to $\TT_i$ by
\begin{align}
 \label{antiT}
  \TT_i^{-1} = \flat \circ \TT_i \circ \flat,
\end{align}
where $\flat$ denotes the anti-involution of the $\Q(v)$-algebra $\tUi$, which fixes the generators $B_i, \K_i$.

For $w\in \widetilde{W}$ with a reduced expression $w =\sigma s_{i_1} \ldots s_{i_r}$ and $\sigma \in \Omega$, we define $\TT_w = \sigma \TT_{i_1} \ldots \TT_{i_r}$, where $\sigma$ acts on $\tUi$ by permuting the indices of generators, $\sigma(B_i) =B_{\sigma i}, \sigma(\K_i) =\K_{\sigma i}$, for all $i\in \I$. Hence,  $\TT_w$ is well defined by Lemma~\ref{lem:Ti}. In particular, the standard results such as Lemma~\ref{lem:Braid-Lus} on braid groups associated to $\widetilde W$ apply to $\TT_w$.

%

\begin{lemma}
\label{lem:fixB}
We have $\TT_w (B_i) = B_{w i}$, for $i\in \I$ and $w \in W$ such that $wi \in \I$.
\end{lemma}

\begin{proof}
This statement is well known for quantum groups \cite{Lus94}, and the proof here is adapted from the proof in \cite[Lemma~8.20]{J95} in the usual quantum group setting.

We first prove the lemma in the rank 2 setting. Assume that $w \in \langle s_i, s_j \rangle$, for some $j\in \I$. The case when $s_i s_j =s_j s_i$ is trivial as we must have $wi=i$. Otherwise, $c_{ij}=-1$, $i\in \I$ such that $wi \in \I$ only happen when $w=s_is_j$, and in this case, a direct computation shows $wi=j$ and $\TT_i\TT_j (B_i) =B_j$.

In general, we prove by induction on $\ell(w)$.
We shall assume $\ell(w)>0$ and let $j \in \I$ such that $w \alpha_j<0$ (clearly $j \neq i$). Then, as in the proof of \cite[Lemma~8.20]{J95},  we have a decomposition $w=w'w''$ such that $w'' \in \langle s_i, s_j \rangle$, $\ell(w)=\ell(w') +\ell(w'')$, $w\alpha_i >0,$ and $w\alpha_j<0$. Thus, $w''\alpha_i>0, w''\alpha_j<0$. Following the proof {\em loc. cit.}, $wi\in \I$ implies $w''i\in \I$.

The inductive assumption applied to $w'$ (and $w'' i \in \I$) yields $\TT_{w'} (B_{w'' i}) =B_{w'w''i} =B_{wi}$.  The rank~2 result in a previous paragraph shows $B_{w'' i} =\TT_{w''} B_i$. Therefore, we obtain
$\TT_w (B_i) = \TT_{w'} \TT_{w''} (B_i)
= \TT_{w'}  (B_{w''i})  = B_{w i}.$
\end{proof}

For $i \in \II$, let $\tUi_{[i]}$ be the subalgebra of $\tUi$ generated by
$B_i,  \TT_{\omega'_i}(B_i), \K_i, \K_{\de-\alpha_i}. 
$

\begin{lemma}  \label{lem:TiFix}
Let $i, j\in \II$ be such that $i\neq j$. Then  $\TT_{\omega_i}(x)=x$, for all $x\in\tUi_{[j]}$.
\end{lemma}

\begin{proof}
Clearly, $\TT_{\omega_i}(\K_j)=\K_j$ and $\TT_{\omega_i}(\K_{\de-\alpha_j})=\K_{\de-\alpha_j}$.
By Lemma \ref{lem:fixB},  $\TT_{\omega_i}(B_j)=B_j$.

By Lemma \ref{lem:Braid-Lus}(b) and noting $\TT_j^{-1}(B_i) =\TT_i(B_j)$, we have
\begin{align}
\TT_{\omega_i}^{-1} \TT_{\omega_j}(B_j)=&\TT_i^{-1}\TT_{\omega_i}\TT_i^{-1}(B_j)
  \label{eq:TTB1} \\
=&\TT_i^{-1} \TT_{\omega_j} \TT_j^{-1}(B_i)
=\TT_{\omega_j}\TT_i^{-1}\TT_j^{-1}(B_i)
=\TT_{\omega_j}(B_j).
\notag
\end{align}
On the other hand, we have
\begin{align}   \label{eq:TTB2}
\TT_{\omega_j}(B_j)=\TT_{\omega_j'}\TT_j(B_j)=\TT_{\omega_j'}(B_j\K_j^{-1})= \TT_{\omega_j'}(B_j)\K_{\de-\alpha_j}.
\end{align}
If follows by \eqref{eq:TTB1}--\eqref{eq:TTB2} that $\TT_{\omega_i}^{-1} \TT_{\omega_j'}(B_j)=\TT_{\omega_j'}(B_j).$
As $\TT_{\omega_i}$ fixes all generators of $\tUi_{[j]}$, the lemma follows.
\end{proof}

\begin{lemma}
\label{lem:T-T}
We have
$\TT_{\omega_i'}(B_i) =\TT_{\omega_i'}^{-1}(B_i),$ for $i\in \II$.
\end{lemma}

\begin{proof}
Note that Lemma~\ref{lem:Braid-Lus} remains valid when the $T$'s therein are replaced by $\TT$'s. Then we have $\TT_{\omega_i'}=\TT_{\omega_i'}^{-1}\prod_{j\neq i} \TT_{\omega_j}^{-c_{ij}}$, thanks to $\TT_{\omega_i} =\TT_{\omega_i'} \TT_i.$ Therefore, it follows by Lemma~\ref{lem:TiFix} that
$\TT_{\omega_i'}(B_i) =\TT_{\omega_i'}^{-1}\prod_{j\neq i} \TT_{\omega_j}^{-c_{ij}} (B_i) =\TT_{\omega_i'}^{-1}(B_i).$
\end{proof}

Here are some additional useful formulas:
\begin{align}
\TT_i \big( \TT_{\omega_i'}(B_i) \big)
&= \TT_{\omega_i'}(B_i) B_i^{(2)} -v_i B_i \TT_{\omega_i'}(B_i) B_i +v_i^2 B_i^{(2)}  \TT_{\omega_i'}(B_i) + \TT_{\omega_i'}(B_i) \K_i,
\label{eq:TTiX}
\\
\TT_i^{-1} \big( \TT_{\omega_i'}(B_i) \big)
&= B_i^{(2)} \TT_{\omega_i'}(B_i) -v_i B_i \TT_{\omega_i'}(B_i) B_i +v_i^2 \TT_{\omega_i'}(B_i) B_i^{(2)} + \TT_{\omega_i'}(B_i) \K_i,
\label{eq:TTiX2}
\\
\TT_i \big([\TT_{\omega_i'}(B_i), B_i]_{v_i^{-2}} \big) &=[B_i, \TT_{\omega_i'}(B_i)]_{v_i^{-2}}.
  \label{TT-imagine}
\end{align}
In the rank 1 case, formulas \eqref{eq:TTiX}--\eqref{eq:TTiX2} reduce to \eqref{T1B0}--\eqref{T1B0-2} while \eqref{TT-imagine} reduces to \eqref{eq:Tm}. The formulas \eqref{eq:TTiX}--\eqref{TT-imagine} are obtained by an Ansatz with the corresponding formulas \eqref{eq:TiX}--\eqref{T-imagine} in the setting of affine quantum groups. Formulas \eqref{eq:TTiX} and \eqref{eq:TTiX2} are equivalent to each other in the presence of \eqref{TT-imagine} or in light of \eqref{antiT}; Equation~\eqref{TT-imagine} is compatible with \eqref{antiT} as well. These formulas shall have interpretations in the $\imath$Hall algebras of $\imath$quivers and $\imath$-weighted projective lines. We skip the details.

\begin{lemma}
 \label{lem:iSerre-rank1}
For $i \in \II$, we have
\begin{align}
\sum_{a=0}^3 (-1)^a \qbinom{3}{a}_i B_i^{3-a} \TT_{\omega_i'}(B_i) B_i^{a}
&= -v_i^{-1} [2]_i^2 [B_i, \TT_{\omega_i'}(B_i)] \K_i,
  \label{Se01}\\
\sum_{a=0}^3 (-1)^a \qbinom{3}{a}_i (\TT_{\omega_i'}(B_i))^{3-a} B_i (\TT_{\omega_i'}(B_i))^{a}
&= -v_i^{-1} [2]_i^2 [\TT_{\omega_i'}(B_i), B_i] \K_\de \K_i^{-1}.
  \label{Se01-2}
\end{align}
\end{lemma}

\begin{proof}
Recall $\K_i$ is central in $\tUi$. Using \eqref{eq:TTiX2}--\eqref{TT-imagine}, we compute
\begin{align*}
[2]_i^{-1} \sum_{a=0}^3 (-1)^a \qbinom{3}{a}_i B_i^{3-a} \TT_{\omega_i'}(B_i) B_i^{a}
& = [2]_i^{-1} \Big[ B_i, \big[ B_i, [B_i, \TT_{\omega_i'}(B_i)]_{v_i^2} \big]_1 \Big]_{v_i^{-2}}
\\
&= \big[ B_i, \TT_i^{-1} (\TT_{\omega_i'}(B_i)) - \TT_{\omega_i'}(B_i) \K_i \big]_{v_i^{-2}}
\\
&= \big[ B_i, \TT_i^{-1} (\TT_{\omega_i'}(B_i)) \big]_{v_i^{-2}} - \big[ B_i, \TT_{\omega_i'}(B_i) \K_i \big]_{v_i^{-2}}
\\
&= \K_i \TT_i^{-1} \big( [ B_i, \TT_{\omega_i'}(B_i)]_{v_i^{-2}} \big) - \big[ B_i, \TT_{\omega_i'}(B_i) \big]_{v_i^{-2}} \K_i
\\
&= \K_i [\TT_{\omega_i'}(B_i), B_i]_{v_i^{-2}}  - \big[ B_i, \TT_{\omega_i'}(B_i) \big]_{v_i^{-2}} \K_i
\\
&= -v_i^{-1} [2]_i [B_i, \TT_{\omega_i'}(B_i)] \K_i.
\end{align*}
The first formula \eqref{Se01} follows.

Applying $\TT_{\omega_i'}$ to both sides of the first formula \eqref{Se01}, we obtain the second formula \eqref{Se01-2} by  using $\TT_{\omega_i'}(\K_i) =\K_\delta \K_i^{-1}$ and Lemma~\ref{lem:T-T}.
\end{proof}

Let us denote the $q$-Onsager algebra in Definition~\ref{def:Onsager} by $\tUi(\widehat{\mathfrak{sl}}_2)$ in the proposition below, in order to distinguish from $\tUi$ of higher rank.

\begin{proposition}
 \label{prop:rank1iso}
Let $i\in \II$. There is a $\Q(v)$-algebra isomorphism  $\aleph_i: \tUi(\widehat{\mathfrak{sl}}_2) \rightarrow \tUi_{[i]}$, which sends $B_1 \mapsto B_i, B_0 \mapsto \TT_{\omega_i'} (B_i), \K_1 \mapsto \K_i, \K_0 \mapsto \K_\de \K_i^{-1}$.
\end{proposition}

\begin{proof}
It follows by Definition~\ref{def:Onsager} and Lemma~\ref{lem:iSerre-rank1} that $\aleph_i$ is a surjective  homomorphism.

It remains to show that $\aleph_i$ is injective. Let $h$ be the Coxeter number of $\fg$. Recall from \eqref{eq:hfilt}--\eqref{eq:filter1} a filtered algebra structure on $\tUi(\widehat{\mathfrak{sl}}_2)$ given by $|\cdot |_{h-1}$ with $|B_1 |_{h-1} =1, |B_0 |_{h} =h-1, |\K_1 |_{h} = |\K_\delta |_{h} =0$. As $\wt( \TT_{\omega_i'} (B_i) ) =\delta-\alpha_i$, we have $|\TT_{\omega_i'} (B_i)| =h-1$, cf. \eqref{eq:filter}. Hence $\aleph_i$ is a homomorphism of filtered algebras, and it induces a homomorphism of the associated graded:
\[
\aleph_i^{\gr}: \gr \tUi(\widehat{\mathfrak{sl}}_2) \longrightarrow \gr \tUi_{[i]},
\]
which can be identified with
\[
\aleph_i^{-}: \U^-(\widehat{\mathfrak{sl}}_2) \otimes \Q(v)[\K_1^\pm, \K_\delta^\pm]
\longrightarrow \U^- \otimes \Q(v)[\K_i^\pm, \K_\delta^\pm],
\]
sending $F_1 \mapsto F_i, F_0 \mapsto T_{\omega_i'} (F_i)$ and $\K_1 \mapsto \K_i, \K_\de \mapsto \K_\de$.
The homomorphism $\aleph_i^{-}|_{\U^-(\widehat{\mathfrak{sl}}_2)}: \U^-(\widehat{\mathfrak{sl}}_2)  \rightarrow \U^-$ is well known to be injective; cf. \cite{Be94}. It follows that $\aleph_i^-$ and $\aleph_i^{\gr}$ are injective, hence so is $\aleph_i$.
\end{proof}

As in \cite{Be94}, we have (cf. \eqref{eq:TTiX})
\begin{align}
\label{eq:T1Ti}
{\TT_i}_{|\tUi_{[i]}} = \aleph_i \circ \TT_1 \circ \aleph_i^{-1},
\qquad
{\TT_{\omega_i}}_{|\tUi_{[i]}} = \aleph_i \circ \TT_{\omega_1} \circ \aleph_i^{-1}, \quad \text{ for } i\in \II.
\end{align}

Define a sign function
\[
o(\cdot): \II \longrightarrow \{\pm 1\}
\]
such that $o(i) o(j)=-1$ whenever $c_{ij} <0$ (there are clearly exactly 2 such functions).
We define the following elements in $\tUi$, for $i\in \II$, $k\in \Z$ and $m\ge 1$ (compare with the rank 1 formulas \eqref{eq:B1n}--\eqref{eq:dB1} and \eqref{eq:dBB}, where $\TT_{\omega_1} =\dag \TT_1$):
\begin{align}
B_{i,k} &= o(i)^k \TT_{\omega_i}^{-k} (B_i),
  \label{Bik} \\
\acute{\Theta}_{i,m} &=  o(i)^m \Big(-B_{i,m-1} \TT_{\omega_i'} (B_i) +v^{2} \TT_{\omega_i'} (B_i) B_{i,m-1}
\label{Thim1} \\
& \qquad\qquad\qquad\qquad + (v^{2}-1)\sum_{p=0}^{m-2} B_{i,p} B_{i,m-p-2}  \K_{i}^{-1}\K_{\de} \Big),
\notag \\
\dB_{i,m} &=\acute{\Theta}_{i,m} - \sum\limits_{a=1}^{\lfloor\frac{m-1}{2}\rfloor}(v^2-1) v^{-2a} \acute{\Theta}_{i,m-2a}\K_{a\de} -\de_{m,ev}v^{1-m} \K_{\frac{m}{2}\de}.
\label{Thim}
\end{align}
In particular, $B_{i,0}=B_i$.

\subsection{A Drinfeld type presentation}
 \label{subsec:Dr2}

Let $k_1, k_2, l\in \Z$ and $i,j \in \II$. We introduce shorthand notations:
\begin{align}
\begin{split}
S(k_1,k_2|l;i,j)
&=  B_{i,k_1} B_{i,k_2} B_{j,l} -[2] B_{i,k_1} B_{j,l} B_{i,k_2} + B_{j,l} B_{i,k_1} B_{i,k_2},
  \\
\SS(k_1,k_2|l;i,j)
&= S(k_1,k_2|l;i,j)  + \{k_1 \leftrightarrow k_2 \}.
  \label{eq:Skk}
  \end{split}
\end{align}
Here and below, $\{k_1 \leftrightarrow k_2 \}$ stands for repeating the previous summand with $k_1, k_2$ switched, so the sums over $k_1, k_2$ are symmetrized.
We also denote
\begin{align}
\begin{split}
R(k_1,k_2|l; i,j)
&=   \K_i  C^{k_1}
\Big(-\sum_{p\geq0} v^{2p}  [2] [\Theta _{i,k_2-k_1-2p-1},\y_{j,l-1}]_{v^{-2}}C^{p+1}
  \label{eq:Rkk} \\
&\qquad\qquad -\sum_{p\geq 1} v^{2p-1}  [2] [\y_{j,l},\Theta _{i,k_2-k_1-2p}]_{v^{-2}} C^{p}
 - [\y_{j,l}, \Theta _{i,k_2-k_1}]_{v^{-2}} \Big),
  \\
 \R(k_1,k_2|l; i,j) &= R(k_1,k_2|l;i,j) + \{k_1 \leftrightarrow k_2\}.
 \end{split}
\end{align}
Sometimes, it is convenient to rewrite part of the summands in \eqref{eq:Rkk} as
\begin{align*}
&-\sum_{p\geq 1} v^{2p-1}  [2] [\y_{j,l},\Theta _{i,k_2-k_1-2p}]_{v^{-2}} C^{p}
 - [\y_{j,l}, \Theta _{i,k_2-k_1}]_{v^{-2}}\\
 =&
-\sum_{p\geq 0} v^{2p-1}  [2] [\y_{j,l},\Theta _{i,k_2-k_1-2p}]_{v^{-2}} C^{p}
+v^{-2}[\y_{j,l}, \Theta _{i,k_2-k_1}]_{v^{-2}}.
\end{align*}
We shall often omit $i,j$ and write $S(k_1,k_2|l)= S(k_1,k_2|l;i,j)$, $\SS(k_1,k_2|l)= \SS(k_1,k_2|l;i,j)$, and similarly for $R$ and $\R$, whenever $i,j$ are clear from the context.

\begin{definition}
\label{def:iDR}
Let $\tUiD$ be the $\Q(v)$-algebra  generated by $\K_{i}^{\pm1}$, $C^{\pm1}$, $H_{i,m}$ and $\y_{i,l}$, where  $i\in \II$, $m \in \Z_{\geq1}$, $l\in\Z$, subject to the following relations, for $m,n \in \Z_{\geq1}$ and $k,l, k_1, k_2 \in \Z$:
\begin{align}
& \K_i, C \text{ are central, }\quad
[H_{i,m},H_{j,n}]=0, \quad \K_i\K_i^{-1}=1, \;\; C C^{-1}=1,
 \label{iDR1}
\\
&[H_{i,m},\y_{j,l}]=\frac{[mc_{ij}]}{m} \y_{j,l+m}-\frac{[mc_{ij}]}{m} \y_{j,l-m}C^m,
\label{iDR2}
\\
&[\y_{i,k} ,\y_{j,l}]=0,   \text{ if }c_{ij}=0,  \label{iDR4}
\\
&[\y_{i,k}, \y_{j,l+1}]_{v^{-c_{ij}}}  -v^{-c_{ij}} [\y_{i,k+1}, \y_{j,l}]_{v^{c_{ij}}}=0, \text{ if }i\neq j,
\label{iDR3a}
 \\ 
&[\y_{i,k}, \y_{i,l+1}]_{v^{-2}}  -v^{-2} [\y_{i,k+1}, \y_{i,l}]_{v^{2}}
=v^{-2}\Theta_{i,l-k+1} C^k \K_i-v^{-4}\Theta_{i,l-k-1} C^{k+1} \K_i
\label{iDR3b} \\
&\qquad\qquad\qquad\qquad\qquad\qquad\quad\quad\quad
  +v^{-2}\Theta_{i,k-l+1} C^l \K_i-v^{-4}\Theta_{i,k-l-1} C^{l+1} \K_i, \notag
\\
  \label{iDR5}
&    \SS(k_1,k_2|l; i,j) =  \R(k_1,k_2|l; i,j), \text{ if }c_{ij}=-1.
\end{align}
Here we set
\begin{align}  \label{Hm0}
{\Theta}_{i,0} =(v-v^{-1})^{-1}, \qquad {\Theta}_{i,m} =0, \; \text{ for }m<0,
\end{align}
and $H_{i,m}$ are related to $\Theta_{i,m}$ by the following equation:
\begin{align}
\label{exp h}
1+ \sum_{m\geq 1} (v-v^{-1})\Theta_{i,m} u^m  = \exp\Big( (v-v^{-1}) \sum_{m\geq 1} H_{i,m} u^m \Big).
\end{align}
\end{definition}
Let us mention that in spite of its appearance the RHS of \eqref{iDR3b} typically has two nonzero terms, thanks to the convention \eqref{Hm0}.

The $\Q(v)$-algebra $\tUiD$ clearly exhibits certain symmetries as follows.

\begin{lemma}  \label{lem:aut}
For each $i\in \II$, we have an automorphism $\t_i$ of the algebra $\tUiD$ given by
\[
\t_i (B_{j,r}) =B_{j,r -\delta_{i,j}}, \quad
\t_i (H_{j,m}) =H_{j,m}, \quad
\t_i (\K_j) = \K_j C^{-\delta_{ij}}, \quad
\t_i(C) =C,
\]
(and hence $\t_i (\Theta_{j,m}) = \Theta_{j,m}$), for all $j\in \II, r\in \Z, m\ge 1$. Moreover, $\t_i \t_k =\t_k\t_i$ for all $i,k \in \II$.
\end{lemma}

\begin{proof}
Follows by inspection on the defining relations for $\tUiD$ in Definition~\ref{def:iDR}.
\end{proof}

Define the generating functions, for $i\in \II$:
\begin{align}
  \label{eq:Genfun}
\begin{cases}
&\Y_{i}(z) = \sum\limits_{k\in\Z} \y_{i,k}z^{k}, \\
 &\bTH_{i}(z) = 1+ \sum\limits_{m>0} (v-v^{-1})\Theta_{i,m}z^{m}, 
\\
& \bDel(z) = \sum\limits_{k\in\Z}  C^k z^k.
\end{cases}
\end{align}

The following is a generalization of Proposition~\ref{prop:equiv4}, which follows by a similar formal manipulation of generating functions as in \S\ref{subsec:proof}; we skip the detail.

\begin{proposition}
   \label{prop:equivij}
The identity \eqref{iDR2} is equivalent to either of the following identities, for $m \ge 1, l\in \Z$ and $i,j \in \II$:
\begin{align}
 \bTH_i (z)  \bB_j(w)
& =   \frac{(1 -v^{-c_{ij}}zw^{-1}) (1 -v^{c_{ij}} zw C)}{(1 -v^{c_{ij}}zw^{-1})(1 -v^{-c_{ij}}zw C)}
  \bB_j(w) \bTH_i (z),
\\
 [\Theta_{i,m},B_{j,l}]+[\Theta_{i,m-2},B_{j,l}]C &=v^{c_{ij}}[\Theta_{i,m-1},B_{j,l+1}]_{v^{-2c_{ij}}}+v^{-c_{ij}} [\Theta_{i,m-1},B_{j,l-1}]_{v^{2c_{ij}}}C.
\end{align}
\end{proposition}

Recall the elements $B_{i,k},  \Theta_{i,m}$ in $\tUi$ defined in \eqref{Bik} and \eqref{Thim}, while elements in the same notation are generators for the algebra $\tUiD$. Below is the main result of this paper.

\begin{theorem}
\label{thm:ADE}
There is a $\Q(v)$-algebra isomorphism ${\Phi}: \tUiD \rightarrow\tUi$, which sends
\begin{align}   \label{eq:map}
\y_{i,k}\mapsto B_{i,k},  \quad H_{i,m}\mapsto H_{i,m}, \quad \Theta_{i,m}\mapsto \Theta_{i,m}, \quad
\K_i\mapsto \K_i, 
\quad C\mapsto \K_\de,
\end{align}
for $i\in \II, k\in \Z$, and $m \ge 1$.
\end{theorem}

The proof of Theorem~\ref{thm:ADE} will be given in \S\ref{subsec:inj} and Section~\ref{sec:relation1} below.

Note that the algebra $\tUiD$ is $\Z \I$-graded by letting
\begin{align}
 \label{eq:wtgrading}
\wt( B_{i,k}) =\alpha_i +k\de, \quad
\wt (\Theta_{i,m}) = m\delta, \quad
\wt (\K_i) =2\alpha_i, \quad
\wt (C) =2\delta,
\end{align}
for $i \in \II, k\in \Z, m\ge 1.$ (It follows that $\wt (\K_0) =2\alpha_0$.) Moreover, $\Phi$ preserves the $\Z \I$-gradings.

\subsection{Proof of the main isomorphism}
\label{subsec:inj}

The details of the proof of Proposition~\ref{thm:iDB hom} below will occupy Section~\ref{sec:relation1}.

\begin{proposition}
\label{thm:iDB hom}
There is a homomorphism ${\Phi}: \tUiD \rightarrow\tUi$ as prescribed in \eqref{eq:map}.
\end{proposition}

\begin{proof}
By Propositions~\ref{prop:iDR4}, \ref{prop:iDR3a}, \ref{prop:iDR2}, \ref{prop:iDR31}, \ref{prop:iDR1} and
\ref{prop:iDR25} in a later Section~\ref{sec:relation1}, all the defining relations \eqref{iDR1}--\eqref{iDR5} for the generators in $\tUiD$ are satisfied by their images under $\Phi$ in $\tUi$.
We shall specify precisely what relations are used when deriving a given relation, to make sure our proofs of these propositions do not run circularly.
Hence ${\Phi}: \tUiD \rightarrow\tUi$ is a homomorphism.
\end{proof}

Assuming Proposition~\ref{thm:iDB hom} (whose proof is much more difficult), we can now complete the proof of Theorem \ref{thm:ADE}.

\begin{proof}[Proof of Theorem \ref{thm:ADE}]
We first show that $\Phi: \tUiD \rightarrow\tUi$ is surjective. All generators for $\tUi$ except $B_0$ are clearly in the image of $\Phi$, and so it remains to show that $B_0\in \text{Im} (\Phi)$. We shall adapt and modify the arguments in the proof of \cite[Theorem~12.11]{Da12}.

The automorphisms $\t_i \in \Aut (\tUiD)$ (see Lemma~\ref{lem:aut}) and $\T_{\omega_i} \in \Aut (\tUi)$, for $i\in \II$, satisfy
\[
\Phi \circ \t_i =\T_{\omega_i} \circ \Phi, \qquad \text{ for } i\in \II.
\]
It follows that $\text{Im}(\Phi)$ is $\T_{\omega_i}$-stable, for each $i \in \II$.

Recall $\theta$ is the highest root in $\cR^+_0$. We can choose and fix ${\bf i} \in \II$ such that
\begin{align}  \label{theta}
\alpha_0 + \alpha_{\bf i} \in \cR^+, \quad
\theta_{\bf i} =\theta -\alpha_{\bf i} \in \cR^+_0, \quad
\text{ and  } s_{\theta_{\bf i}}(\alpha_{\bf i})=\theta.
\end{align}
Write $t_{\omega_{\bf i}} =\sigma_{\bf i} s_{i_1} \ldots s_{i_N}$ with $\sigma_{\bf i} \in \Omega$. Then $t_{\omega_{\bf i}}(\theta) =\theta-\de\in -\cR^+$, and so there exists $p$ such that $s_{i_N} \ldots s_{i_{p+1}} (\alpha_{i_p}) =\theta$. Hence, $y:=\T_{i_N}^{-1} \ldots \T_{i_{p+1}}^{-1} (B_{i_p}) \in  \text{Im} (\Phi)$; note $\wt (y) =\theta$.

Consider $\T_{\omega_{\bf i}} (y) =  \T_{\sigma_{\bf i}} \T_{i_1} \ldots \T_{i_{p-1}} (B_{i_p}) \in \text{Im} (\Phi)$, since $\text{Im}(\Phi)$ is $\T_{\omega_{\bf i}}$-stable. Note that $\sigma_{\bf i} s_{i_1} = s_0 \sigma_{\bf i}$ (see \cite[Lemma 3.1]{Be94}), and so $\T_{s_0\omega_{\bf i}} =\T_0^{-1} \T_{\omega_{\bf i}}$.
Since $\wt (\T_0^{-1} \T_{\omega_{\bf i}} (y) )=\alpha_0$ (i.e., $s_0\omega_{\bf i} s_{i_N} \ldots s_{i_{p+1}} (\alpha_{i_p}) =\alpha_0$), it follows by Lemma~\ref{lem:fixB} that $\T_0^{-1} \T_{\omega_{\bf i}} (y) =B_0$, and hence $\T_{\omega_{\bf i}} (y) =\T_0(B_0) =\K_{0} B_0$. Therefore, we have $\K_{0} B_0\in \text{Im} (\Phi)$ and thus $B_0\in \text{Im} (\Phi)$.
\vspace{2mm}


It remains to show that the map $\Phi: \tUiD \rightarrow\tUi$ is injective. Let us explain the simple underlying idea of the arguments for injectivity: we shall first show that $\Phi$ on the associated graded level when restricted to ``a positive half" of $\tUi$ (which correspond to ``a quarter" in the affine quantum group $\U$, for which the roots share the same sign in Kac-Moody sense and in the Drinfeld current sense) is injective; this is summarized by the commutative diagram \eqref{diag:1} below. Then we show the injectivity on ``the positive half" implies the injectivity of $\Phi: \tUiD \rightarrow\tUi$ fully via the translation automorphisms $\t_i$.

Denote by $\tUi_>$ (respectively, $\tUiD_>$) the subalgebra of $\tUi$ (respectively, $\tUiD$) generated by $B_{i,m}$, $H_{i,m}, \K_i$, for $m\ge 1$, and $i\in \II$. Then ${\Phi}: \tUiD \longrightarrow\tUi$ restricts to a surjective homomorphism ${\Phi}: \tUiD_> \longrightarrow\tUi_>$.

Define a filtration on $\tUiD_>$ by
\begin{align}  \label{eq:filt1D}
(\tUiD_>)^0 \subset (\tUiD_>)^1 \subset \cdots \subset (\tUiD_>)^m \subset \cdots
\end{align}
by setting
\begin{align}  \label{eq:filtD}
(\tUiD_>)^m &=\Q(v)\text{-span} \big\{x=B_{i_1,m_1} B_{i_2,m_2} \ldots B_{i_r,m_r} \Theta_{j_1,n_1} \Theta_{j_2,n_2} \ldots \Theta_{j_s,n_s} \K_\mu
\\
&\quad
\mid \mu \in \N\I, i_1, \ldots, i_r, j_1, \ldots j_s, \in \II, m_1,\ldots, m_r, n_1, \ldots, n_s \ge 1, \hgt^+(x)\leq m \big\}.
\notag
\end{align}
Here we have denoted
\begin{align}  \label{eq:ht}
\hgt^+(x) :=\sum_{a=1}^r \hgt(m_a\de +\alpha_{i_a}) +\sum_{b=1}^s n_b \hgt(\de),
\end{align}
where $\hgt(\beta)$ denotes the height of a positive root $\beta$; compare with the $\N\I$-grading on $\tUi$ by \eqref{eq:wtgrading}. Recalling $\widetilde{\U}^{\imath,0}$ from \eqref{eq:UiCartan}, we have
\[
(\tUiD_>)^0 = \widetilde{\U}^{\imath,0} =\Q(v) [ \K_i^{\pm 1} \mid i\in \I ].
\]
The filtration \eqref{eq:filt1D}--\eqref{eq:filtD} on $\tUiD_>$ defined via a height function is compatible with the filtration \eqref{eq:filt1}--\eqref{eq:filt} on $\tUi$ under $\Phi$, and thus the surjective homomorphism ${\Phi}: \tUiD_> \longrightarrow\tUi_>$ induces a surjective homomorphism
\begin{align}  \label{eq:gradeP}
{}^{\text{gr}} {\Phi}_>: \tUiDgr_> \longrightarrow\tUigr_>.
\end{align}

Recall from \eqref{eq:phi1} an isomorphism $\phi:  {}^{\text{Dr}}\U \rightarrow \U$ for the affine quantum group $\U$. Denote by ${}^{\text{Dr}}\U^-_<$ the $\Q(v)$-subalgebra of $\U^-$ generated by $x^-_{i,-k}$, for $i\in \II, k>0$, and denote by $\U^-_< =\phi({}^{\text{Dr}}\U^-_<)$. Then $\phi$ restricts to an isomorphism
\begin{align}  \label{eq:phi}
\phi:  {}^{\text{Dr}}\U^-_< \stackrel{\cong}{\longrightarrow} \U^-_<.
\end{align}

On the other hand, consider the associated graded $\tUigr$ with respect to the filtration on $\tUi$ given by \eqref{eq:filt1}--\eqref{eq:filt}. Recall the algebra isomorphism in \eqref{eq:filter}
\begin{align*}
\mathbb G: \U^- \otimes \widetilde{\U}^{\imath,0} \longrightarrow \tUigr,
\qquad
F_i \mapsto \overline{B}_i, \quad
\K_i \mapsto \overline{\K}_i,
\end{align*}
where $\U^-$ denotes half a Drinfeld-Jimbo quantum group generated by $F_i$, for $i\in \I$.
The homomorphism $\mathbb G$ above restricts to an isomorphism
\begin{align}  \label{eq:grade1}
\mathbb G: \U^-_< \otimes \widetilde{\U}^{\imath,0} \stackrel{\cong}{\longrightarrow} \tUigr_>.
\end{align}
Finally, by definition \eqref{eq:filtD}--\eqref{eq:ht} of the filtration on $\tUiD_>$, we have a surjective homomorphism
\begin{align}  \label{eq:Xi}
\Xi: {}^{\text{Dr}}\U^-_< \otimes \widetilde{\U}^{\imath,0}
 \longrightarrow
 \tUiDgr_>,
\end{align}
which sends $x_{i,-k}^- \mapsto \ov{B}_{i,k}$, for $k>0$ (note the opposite sign in indices).

Collecting \eqref{eq:gradeP}, \eqref{eq:grade1}, \eqref{eq:Xi} and \eqref{eq:phi}, we obtain the following commutative diagram
\begin{align}  \label{diag:1}
\xymatrix{
 {}^{\text{Dr}}\U^-_< \otimes \widetilde{\U}^{\imath,0}
 \ar[rr]^{\Xi}
\ar[d]^{\phi,\cong}
 && \tUiDgr_>
\ar[d]^{{}^{\text{gr}}\Phi_>}
\\
\U^-_< \otimes \widetilde{\U}^{\imath,0}
\ar[rr]^{\mathbb G, \cong}
&& \tUigr_>   }
\end{align}
Since the homomorphisms $\Xi$ and ${}^{\text{gr}}\Phi_>$ are surjective while $\phi$ and $\mathbb G$ are isomorphisms, we conclude that ${}^{\text{gr}}\Phi_>: \tUiDgr_> \longrightarrow\tUigr_>$ is injective (and an isomorphism).

Now we show that ${}^{\text{gr}}\Phi: \tUiDgr \longrightarrow\tUigr$ is injective (and hence an isomorphism); a similar argument was used in \cite[Proposition~5.2]{Da15}. Recall $\rho =\sum_{i\in \II} \omega_i$ is half the sum of positive roots in $\Phi$, and thus $\T_\rho =\prod_{i\in \II} \T_{\omega_i}$; the automorphism $\T_\rho$ on $\tUi$ induces an automorphism (with the same notation) on $\tUigr$.
Assume that a finite linear combination $X =\sum (*) B_{i_1,r_1} B_{i_2,r_2} \ldots B_{i_t,r_t} \Theta_{j_1,n_1} \Theta_{j_2,n_2} \ldots \Theta_{j_s,n_s} \K_\mu$ (with $r_a\in \Z, n_b \ge 1$, for various $a, b$) lies in the kernel of ${}^{\text{gr}}\Phi$. Recall the automorphisms $\t_i$ of $\tUiD$ from Lemma~\ref{lem:aut}. Applying an automorphism $\prod_{i\in \II} \t_{i}^{-N}$ to $X$ produces an element $X_N =\prod_{i\in \II} \t_{i}^{-N}(X)$, which is obtained from $X$ with all indices $r_a$ of $B$'s in each summand of $X$ shifted to $r_a+N$. Pick and fix an $N$ large enough so that all relevant $r_a+N$ are positive, that is, $X_N \in \tUiDgr_>$. Thanks to the following commutative diagram
\begin{align*}
\xymatrix{
\tUiDgr_>
 \ar[rr]^{\prod_{i\in \II} \t_i^{N}}
\ar[d]^{{}^{\text{gr}}\Phi_>}
 && \tUiDgr
\ar[d]^{{}^{\text{gr}}\Phi}
\\
\tUigr_>
\ar[rr]^{{\T}_{\rho}^N}
&& \tUigr  }
\end{align*}
We have $\T_\rho^N \circ {}^{\text{gr}}\Phi_> (X_N) = {}^{\text{gr}}\Phi \circ (\prod_{i\in \II} \t_i^{N}) (X_N) = {}^{\text{gr}}\Phi (X) =0$, and hence ${}^{\text{gr}}\Phi_> (X_N) =0$. Since ${}^{\text{gr}}\Phi_>: \tUiDgr_> \longrightarrow\tUigr_>$ is injective, we must have $X_N=0$ and hence $X=0$.

This proves the injectivity of ${}^{\text{gr}}\Phi$. It follows that $\Phi: \tUiD \longrightarrow\tUi$ is injective. This completes the proof of Theorem \ref{thm:ADE}.
\end{proof}

\begin{remark}
It follows by Lemma~\ref{lem:fixB} that $B_0=\T_{\theta_{\bf i}} \T_{\omega_{\bf i}'}(B_{\bf i})= o({\bf i}) \K_0\T_{\theta_{\bf i}} (B_{{\bf i},-1})$, where $\bf i \in \II$ is chosen as in \eqref{theta}.
The inverse ${\Phi}^{-1}: \tUi \rightarrow \tUiD$ to the isomorphism $\Phi$ in \eqref{eq:map} is given by
\begin{align*}
\K_j \mapsto \K_j, &\quad \K_0\mapsto  C \K_\theta^{-1},
\quad
B_{j}\mapsto   \y_{j,0}, \quad B_0\mapsto  o({\bf i}) \TT_{\theta_{\bf i}} (\y_{{\bf i},-1})C \K_\theta^{-1},
\quad \text{for } j\in \II.
\end{align*}
Another formula for $\Phi^{-1}(B_0)$ can be read off from the proof of Theorem \ref{thm:ADE}.
\end{remark}

\begin{remark}
One can construct all real $v$-root vectors in $\tUi$ with the help of braid group action. Together with the imaginary $v$-root vectors which have been constructed, one can write down a natural PBW basis for $\tUi$, following \cite{BCP99}.
\end{remark}

\begin{remark}
Definition~\ref{def:iDR} for $\tUiD$ formally makes sense for a generalized Cartan matrix (GCM) $(c_{ij})_{i,j\in \I}$ of a simply-laced Kac-Moody algebra $\fg$, and hence can be regarded as a definition of {\em $\imath$quantum loop Kac-Moody algebras} for $\fg$. A most interesting subclass of these new algebras will be {\em $\imath$quantum toroidal algebras} when $\fg$ is of affine type. Once we have extended Definition~\ref{def:iDR} to a wide class of affine $\imath$quantum groups, similar relaxing of conditions on GCMs will allow us to define more general {\em $\imath$quantum loop Kac-Moody algebras}.
\end{remark}

\begin{remark}
There is a nonstandard comultiplication $\Delta^{\text{n-std}}$ (due to Drinfeld) on the affine quantum group ${}^{\text{Dr}} \U$ (or its Drinfeld double $\tU$) via its Drinfeld presentation. It will be interesting to ask if there is a natural coideal subalgebra of ${}^{\text{Dr}} \U$ with respect to $\Delta^{\text{n-std}}$ (which, if it exists, could be different from $\tUiD$).
\end{remark}

\subsection{Classical limit}
    \label{subsec:classical}

Denote the Chevalley generators of the semisimple (or even Kac-Moody) Lie algebra $\fg$ over $\Q$ by $e_i, f_i, h_i$, for $i\in \II$.
Denote $L\fg =\fg \otimes \Q[t,t^{-1}]$, and the affine Lie algebra (as a vector space) $\widehat{\fg} = L\fg \oplus \Q c$; set $x_k =x\otimes t^k$, for $x\in \fg, k\in \Z$.  Denote by $\omega$ the Chevalley involution on $\fg$ such that $\omega(c)= -c, \omega(e_i)= -f_i, \omega(f_i) = -e_i, \omega(h_i)= -h_i$, for all $i$. (More generally, one can take $\omega =\omega_a$, for any fixed nonzero scalar $a$, such that  $\omega_a(c)= -c, \omega_a(e_i)= -af_i, \omega_a(f_i) = -a^{-1} e_i, \omega_a(h_i)= -h_i$.) Then $\omega$ induces an involution $\widehat{\omega}$ on $\widehat{\fg}$ such that $\widehat{\omega} (x_k) =\omega(x)_{-k}$, for all $x\in \fg, k\in \Z$. The algebra $\tUi$ of split affine type ADE specializes at $v=1$ to the enveloping algebra of the ${\widehat{\omega}}$-fixed point subalgebra $(L\fg)^{\widehat{\omega}}$.

Let us examine in detail $(L\fg)^{\widehat{\omega}}$ in the case when $\fg =\sll_2$ with standard basis $\{e,h,f\}$. Set $b_r :=f_r + e_{-r}, t_r :=h_r - h_{-r}$, for $r\in \Z$. Note $t_{-r} = -t_r$ and $t_0=0$. Then $\{b_r, t_m\mid r\in \Z, m\ge 1\}$ forms a basis for $(L\sll_2)^{\widehat{\omega}}$. They satisfy the relations, for $r, s \in \Z, m\ge 1$,
\begin{align}
[b_r, b_s] &= t_{s-r},
\label{eq:bb1}
\\
[t_m, t_n] &=0,
 \label{eq:tt} \\
[t_m, b_r] &=-2 b_{m+r} + 2b_{-m+r},
 \label{eq:tb} \\
[b_r, b_{s+1}] - [b_{r+1}, b_s] &= t_{s-r+1} -t_{s-r-1}.
\label{eq:bb2}
\end{align}
Clearly, \eqref{eq:bb1}--\eqref{eq:tb} are defining relations for  $(L\sll_2)^{\widehat{\omega}}$.
One checks that \eqref{eq:bb1} and \eqref{eq:bb2} are equivalent. Hence \eqref{eq:tt}--\eqref{eq:bb2} are also defining relations for $(L\sll_2)^{\widehat{\omega}}$, providing a non-standard presentation for the Onsager algebra, which is compatible with our Drinfeld type presentation up to suitable changes of indices: $b_r \leftrightarrow B_{1,-r}$ and $t_r \leftrightarrow H_{r}$; this index issue originates from the identification of the associated graded of the $q$-Onsager algebra with the positive (instead of the negative) half of quantum loop $\sll_2$ in \cite{BK20} (which is followed in our Section~\ref{sec:Onsager}).

A Drinfeld type presentation of $(L\fg)^{\widehat{\omega}}$ in ADE case (generalizing \eqref{eq:tt}--\eqref{eq:bb2} for $\fg=\sll_2$) can be similarly written down, compatible with Definition~\ref{def:iDR}.

\section{Verification of Drinfeld type new relations}
  \label{sec:relation1}

In this section, we prove that ${\Phi}: \tUiD \rightarrow\tUi$ defined by \eqref{eq:map} is a homomorphism. We shall first establish the relations \eqref{iDR1}, \eqref{iDR4}--\eqref{iDR3b}, and an easier case of \eqref{iDR2}  in $\tUi$. We then establish the more challenging relations \eqref{iDR5} and \eqref{iDR2} (when $c_{ij}=-1$) in $\tUi$.
\subsection{Relation \eqref{iDR4}}

\begin{proposition}
 \label{prop:iDR4}
Assume $c_{ij}=0$, for $i,j \in \II$. Then
$[B_{i,k},B_{j,l}]=0,$ for all $k,l\in\Z$.
\end{proposition}

\begin{proof}
The identity for $k=l=0$, i.e., $[B_{i},B_{j}]=0,$ is exactly the defining relation \eqref{eq:S1} for $\tUi$.
The identity for general $k,l$ follows by applying $\TT_{\omega_i}^{-k} \TT_{\omega_j}^{-l}$ to the above identity and using Lemma~\ref{lem:Braid-Lus}(c) and Lemma~\ref{lem:TiFix}.
\end{proof}
\subsection{Relation \eqref{iDR3a}}

\begin{lemma} [\text{cf. \cite[Lemma 3.3]{Be94}}]
 \label{lem:TXij}
For $i\neq j\in\II$ such that $c_{ij}=-1$, denote
\begin{align*}
X_{ij}:=\TT_j^{-1}B_i= B_jB_i-vB_iB_j.
\end{align*}
Then we have
$
\TT_{\omega_i}(X_{ji})=\TT_{\omega_j}(X_{ij}).
$
\end{lemma}

\begin{proof}
This follows by a direct computation using Lemma~\ref{lem:Braid-Lus} and Lemma~\ref{lem:TiFix}:
\begin{align*}
\TT_{\omega_j}(X_{ij})= &\TT_{\omega_j} \TT_j^{-1}(B_i)= \TT_j\TT_{\omega_j}^{-1}\TT_{\omega_i}(B_i)
=\TT_{\omega_i}\TT_j(B_i)=\TT_{\omega_i}(X_{ji}).
\end{align*}
\end{proof}

Now we are ready to establish the relation \eqref{iDR3a}.
\begin{proposition}
 \label{prop:iDR3a}
We have $[\y_{i,k}, \y_{j,l+1}]_{v^{-c_{ij}}}  -v^{-c_{ij}} [\y_{i,k+1}, \y_{j,l}]_{v^{c_{ij}}}=0$, for $i\neq j \in \II$ and $k, l \in \Z.$
\end{proposition}

\begin{proof}
If $c_{ij}=0$, then the identity in the proposition follows directly by \eqref{iDR4}.

Assume $c_{ij}=-1$.
Note that
$
 v[B_{i,k+1},B_{j}]_{v^{-1}}=  -  o(i)^{k+1} \TT_{\omega_i}^{-(k+1)}(X_{ij}).
$ 
Hence we have
\begin{align*}
[B_{i,k},B_{j,1}]_{v}&= B_{i,k}B_{j,1}-vB_{j,1}B_{i,k}
 \\
& = o(i)^k o(j) \TT_{\omega_i}^{-k} \TT_{\omega_j}^{-1}(X_{ji})
= - o(i)^{k+1} \TT_{\omega_i}^{-k} \TT_{\omega_i}^{-1} (X_{ij})
 \\
&= - o(i)^{k+1} \TT_{\omega_i}^{-(k+1)}(X_{ij})=v[B_{i,k+1},B_{j}]_{v^{-1}}.
\end{align*}
So we have obtained an identity
$[B_{i,k},B_{j,1}]_{v}-v[B_{i,k+1},B_{j}]_{v^{-1}}=0.$
The identity in the proposition follows by applying $\TT_{\omega_j}^{-l}$ to this identity.
\end{proof}
\subsection{Relation \eqref{iDR2} for $c_{ij}=0$}

We shall identify $C=\K_\de$ below.

\begin{lemma}
  \label{lem:comb}
For $j\in\II$, we have
\begin{align*}
\Theta_{j,n}=
\begin{cases}
v^{-2} C \Theta_{j,n-2}+v^2\K_j^{-1} \big( [B_{j,0},B_{j,n}]_{v^{-2}}+[B_{j,n-1},B_{j,1}]_{v^{-2}} \big), & \text{ if } n\ge 3,
\\
-v^{-1} C +v^2\K_j^{-1} \big( [B_{j,0},B_{j,2}]_{v^{-2}}+[B_{j,1},B_{j,1}]_{v^{-2}} \big), & \text{ if } n=2,
\\
v^2 \K_{j}^{-1} [\y_{j,0}, \y_{j,1}]_{v^{-2}}, & \text{ if } n= 1.
\end{cases}
\end{align*}
In particular, for any $n\ge 1$, the element $\Theta_{j,n}$  is a $\Q(v)[C^{\pm 1},\K_j^{\pm 1}]$-linear combination of $1$ and $[\y_{j,k}, \y_{j,l+1}]_{v^{-2}} +[\y_{j,l}, \y_{j,k+1}]_{v^{-2}}$, for $l, k \in \Z$.
\end{lemma}

\begin{proof}
The recursion formulas in the lemma are reformulations of \eqref{iDR3b} with $k=0$ and $l=n-1$. The second statement follows by an induction on $n$ using the recursion formulas. (A precise linear combination can be written down, but will not be needed.)
%
\end{proof}

\begin{proposition}
  \label{prop:iDR2}
Assume $c_{ij}=0$, for $i,j\in \II$. Then, for $m\geq 1$ and $r \in\Z$, we have
\begin{align*}
[\Theta_{i,m},B_{j,r}] &=0 =[H_{i,m},B_{j,r}].
\end{align*}
\end{proposition}

\begin{proof}
We shall prove the first equality only (as $H_{i,n}$ can be expressed in terms of $\Theta_{i,m}$ for various $m$).
By Lemma~\ref{lem:comb} (with index $j$ replaced by $i$), it suffices to check that $[B_{i,k}, B_{i,l}]_{v^{-2}}$ commutes with $B_{j,r}$ for all $k,l,r$. But this clearly follows by the commutativity between $B_{i,k}$ and $B_{j,r}$ \eqref{iDR4}.
\end{proof}
\subsection{Relations \eqref{iDR3b} and \eqref{iDR1}--\eqref{iDR2} for $i=j$}
 \label{subsec:rank1}

\begin{proposition}
 \label{prop:iDR31}
Relation \eqref{iDR3b} and relations  \eqref{iDR1}--\eqref{iDR2} for $i=j \in \II$ hold in $\tUi$.
\end{proposition}

\begin{proof}
These rank one relations (for a fixed $i\in \II$)  follow by transporting the corresponding relations in $q$-Onsager algebra (see Theorem~\ref{thm:Dr1}) to $\tUi$ using Proposition~\ref{prop:rank1iso}. Note that for \eqref{iDR3b}, the overall sign $o(i)^{k+l+1}$ originated from \eqref{Bik}--\eqref{Thim} cancels out.
\end{proof}
\subsection{Relation \eqref{iDR1} for $i\neq j$}

We shall derive the identity $[H_{i,m},H_{j,n}]=0$ in \eqref{iDR1}, for $i\neq j \in \II$, from the relations \eqref{iDR3b} (proved above) and \eqref{iDR2}. The proof of \eqref{iDR2} in the following subsections will not use the relation \eqref{iDR1}.

\begin{lemma}
  \label{lem:comm}
For $i,j \in \II$, $l, k \in \Z$ and  $m\ge 1$, we have
\[
\Big[ H_{i,m}, [\y_{j,k}, \y_{j,l+1}]_{v^{-2}} +[\y_{j,l}, \y_{j,k+1}]_{v^{-2}} \Big] =0.
\]
\end{lemma}

\begin{proof}
Using \eqref{iDR2} we have
\begin{align*}
 \Big[ H_{i,m},  &  [\y_{j,k}, \y_{j,l+1}]_{v^{-2}} +[\y_{j,l}, \y_{j,k+1}]_{v^{-2}} \Big]
\\\notag
= & [H_{i,m}, \y_{j,k}] \y_{j,l+1} +v^{-2} \y_{j,l+1} [\y_{j,k}, H_{i,m}]
\\\notag
& + [H_{i,m}, \y_{j,l}] \y_{j,k+1} +v^{-2} \y_{j,k+1} [\y_{j,l}, H_{i,m}]
\\\notag
& +\y_{j,l} [H_{i,m}, \y_{j,k+1}]  +v^{-2}  [\y_{j,k+1}, H_{i,m}] \y_{j,l}
\\\notag
& + \y_{j,k} [H_{i,m}, \y_{j,l+1}]  +v^{-2}  [\y_{j,l+1}, H_{i,m}] \y_{j,k}\\\notag
= &\frac{[mc_{ij}]}{m}\Big( (\y_{j, k+m} -\y_{j,k-m} C^m) \y_{j, l+1} -v^{-2} \y_{j,l+1} (\y_{j,k+m} -\y_{j,k-m} C^m)
\\\notag
& + (\y_{j, l+m} -\y_{j,l-m} C^m) \y_{j, k+1} -v^{-2} \y_{j,k+1} (\y_{j,l+m} -\y_{j,l-m} C^m)
\\\notag
& + \y_{j, l}(\y_{j, k+m+1} -\y_{j,k-m+1} C^m)  -v^{-2} (\y_{j,k+m+1} -\y_{j,k-m+1} C^m) \y_{j, l}
\\\notag
& + \y_{j, k}(\y_{j, l+m+1} -\y_{j,l-m+1} C^m)  -v^{-2} (\y_{j,l+m+1} -\y_{j,l-m+1} C^m) \y_{j, k}\Big).
\end{align*}
The above equality can be further rewritten as
\begin{align*}
 &\Big[ H_{i,m}, [\y_{j,k},  \y_{j,l+1}]_{v^{-2}} +[\y_{j,l}, \y_{j,k+1}]_{v^{-2}} \Big]= \frac{[mc_{ij}]}{m}\times
\\
&\Big([\y_{j,k+m}, \y_{j,l+1}]_{v^{-2}} - [\y_{j,k-m}, \y_{j,l+1}]_{v^{-2}}  C^m
   + [\y_{j,l+m}, \y_{j,k+1}]_{v^{-2}} - [\y_{j,l-m}, \y_{j,k+1}]_{v^{-2}}  C^m
\\
&+ [\y_{j,l}, \y_{j,k+m+1}]_{v^{-2}} - [\y_{j,l}, \y_{j,k-m+1}]_{v^{-2}}  C^m
   + [\y_{j,k}, \y_{j,l+m+1}]_{v^{-2}} - [\y_{j,k}, \y_{j,l-m+1}]_{v^{-2}}  C^m\Big).
\end{align*}
There are 4 terms in each of the above 2 lines, and we add up column by column using \eqref{iDR3b}. Note that the sum of column 1 cancels with the sum of column 4 since both are equal to (modulo an opposite sign)
\begin{align*}
\begin{cases}
v^{-2} \Theta_{j, l-k-m+1} C^{k+m} \K_j -v^{-2} \Theta_{j, l-k-m-1} C^{k+m+1} \K_j, & \text{ if } l>k+m,
\\
v^{-2} \Theta_{j, k+m-l+1} C^l \K_j -v^{-2} \Theta_{j, k+m-l-1} C^{l+1} \K_j, & \text{ if } l<k+m,
\\
2 v^{-2} \Theta_{j, 1} C^l \K_j, & \text{ if } l =k+m.
\end{cases}
\end{align*}
Similarly, the sum of column 2 cancels with the sum of column 3 since both are equal to (modulo an opposite sign)
\begin{align*}
\begin{cases}
v^{-2} \Theta_{j, k-l-m+1} C^{l+m} \K_j  -v^{-2} \Theta_{j, k-l-m-1} C^{l+m+1} \K_j, & \text{ if } k>l+m,
\\
v^{-2} \Theta_{j, l+m-k+1} C^k \K_j -v^{-2} \Theta_{j, l+m-k-1} C^{k+1} \K_j , & \text{ if } k<l+m,
\\
2 v^{-2} \Theta_{j, 1} C^k \K_j, & \text{ if } k =l+m.
\end{cases}
\end{align*}
This proves the lemma.
\end{proof}

\begin{proposition}
 \label{prop:iDR1}
Relation \eqref{iDR1} for $i\neq j  \in \II$ in $\tUi$ follows from the relations \eqref{iDR2} and \eqref{iDR3b} in $\tUi$.
\end{proposition}

\begin{proof}
It follows by Lemma~\ref{lem:comb} and Lemma~\ref{lem:comm} that $[ H_{i,m}, \Theta_{j,a}] =0$, for all $m, a\ge 1.$
Since $H_{j,n}$ for any $n \ge 1$ is a linear combination of monomials in $\Theta_{j,a}$, for various $a\ge 1$ by \eqref{exp h}, we conclude that $[ H_{i,m}, H_{j,n}] =0$, whence \eqref{iDR1}.
\end{proof}


%
%
\subsection{Two more relations rephrased}

It remains to establish relations \eqref{iDR5} and \eqref{iDR2} for $c_{ij}=-1$ in $\tUi$, with the help of the finite type Serre relation \eqref{eq:S2} and \eqref{iDR3a}--\eqref{iDR3b}.

By Proposition~\ref{prop:equivij}, the relation \eqref{iDR2} (for $c_{ij}=-1$) is equivalent to the following relation
\begin{align}
\label{iDR2-reform}
 [ & \Theta_{i,k},B_{j,r}]+[\Theta_{i,k-2},B_{j,r}]\K_\de
 \\
&= v^{-1}[\Theta_{i,k-1},B_{j,r+1}]_{v^2}+v[\Theta_{i,k-1},B_{j,r-1}]_{v^{-2}}\K_\de, \quad \text{ for } k \ge 0 \text{ and  } r \in \Z.
\notag
\end{align}
Clearly \eqref{iDR2-reform} also holds for $k< 0$.

The relation \eqref{iDR5} in $\tUi$ (where we can assume $k_2\ge k_1$ without loss of generality) reads:
\begin{align}
\label{iDR5-reform}
\SS(k_1,k_2|l) =\R(k_1,k_2|l), \quad \text{ for } k= k_2 -k_1 \ge 0 \text{ and } l \in \Z.
\end{align}

We shall prove \eqref{iDR2-reform} and \eqref{iDR5-reform} simultaneously and inductively on $k$ in \S\ref{subsec:iDR2=>iDR5}--\ref{subsec:iDR5=>iDR2} below. To that end, we shall refer to \eqref{iDR2-reform} and \eqref{iDR5-reform} as \eqref{iDR2-reform}$_k$ and \eqref{iDR5-reform}$_k$, respectively. We then denote by \eqref{iDR2-reform}$_{\le k}$ (respectively, \eqref{iDR2-reform}$_{< k}$) the identities \eqref{iDR2-reform}$_\ell$ for all $\ell \le k$ (respectively, for all $\ell <k$); and similarly for \eqref{iDR5-reform}$_{\le k}$.

\subsection{Implication from $\eqref{iDR2-reform}_{< k}$ to $\eqref{iDR5-reform}_{\leq k}$ }
\label{subsec:iDR2=>iDR5}

We shall fix $i,j \in \II$ such that $c_{ij}=-1$ throughout this subsection.
We identify $C=\K_\de$ below.

\begin{lemma}  \label{lem:SSSa}
For $k_1, k_2, l \in \Z$, we have
\begin{align*}
& \SS(k_1,k_2+1 |l) + \SS(k_1+1,k_2|l) -[2] \SS(k_1+1,k_2+1 |l-1)
 \\
&= \Big( -[B_{jl}, \Theta_{i, k_2-k_1+1}]_{v^{-2}} C^{k_1} +v^{-2} [B_{jl}, \Theta_{i, k_2-k_1-1}]_{v^{-2}} C^{k_1+1} \Big) \K_i + \{k_1 \leftrightarrow k_2\}.
\end{align*}
\end{lemma}
(The proof of the lemma uses only the relations \eqref{iDR3a}--\eqref{iDR3b}.)

\begin{proof}
We rewrite \eqref{eq:Skk} as
  \begin{align*}
\SS(k_1,k_2|l)
&= B_{i,k_2} [B_{i,k_1}, B_{j,l}]_{v^{-1}} - v [B_{i,k_1}, B_{j,l}]_{v^{-1}} B_{i,k_2}
+\{k_1 \leftrightarrow k_2\}.
\end{align*}
This together with \eqref{iDR3a} implies that
\begin{align}
 &\SS(k_1+1,k_2+1|l-1)
 \label{eq:Skk1a} \\
&=  (B_{i,k_2+1} [B_{i,k_1+1}, B_{j,l-1}]_{v^{-1}} - v [B_{i,k_1+1}, B_{j,l-1}]_{v^{-1}} B_{i,k_2+1})
+\{k_1 \leftrightarrow k_2\}
\notag \\
&= (v^{-1} B_{i,k_2+1} [B_{i,k_1}, B_{j,l}]_{v} -  [B_{i,k_1}, B_{j,l}]_{v} B_{i,k_2+1})
+\{k_1 \leftrightarrow k_2\}.  \notag
\end{align}

Using \eqref{eq:Skk1a}, we compute
\begin{align}
  \label{eq:Skk2a}
& \SS(k_1,k_2+1 |l) + \SS(k_1+1,k_2|l) -[2] \SS(k_1+1,k_2+1|l-1)
\\
&= \Big( \SS(k_1,k_2+1 |l)
-[2] \big(v^{-1} B_{i,k_2+1} [B_{i,k_1}, B_{j,l}]_{v} -  [B_{i,k_1}, B_{j,l}]_{v} B_{i,k_2+1} \big) \Big)
+\{k_1 \leftrightarrow k_2\}
 \notag \\
&=
\Big( \big( B_{i,k_1} [B_{i,k_2+1}, B_{j,l}]_{v} - v^{-1} [B_{i,k_2+1}, B_{j,l}]_{v} B_{i,k_1}
+B_{i,k_2+1} [B_{i,k_1}, B_{j,l}]_{v} - v^{-1} [B_{i,k_1}, B_{j,l}]_{v} B_{i,k_2+1} \big)
\notag \\
& \quad -[2] \big(v^{-1} B_{i,k_2+1} [B_{i,k_1}, B_{j,l}]_{v} -  [B_{i,k_1}, B_{j,l}]_{v} B_{i,k_2+1} \big) \Big)
 + \{k_1 \leftrightarrow k_2\}
 \notag \\
&= \big(B_{jl} [B_{i, k_2+1}, B_{i, k_1}]_{v^2} -v^{-2} [B_{i, k_2+1}, B_{i, k_1}]_{v^2} B_{jl} \big)
 + \{k_1 \leftrightarrow k_2\},
 \notag
\end{align}
where the last identity is obtained by first combining the third and fifth terms (and respectively, the fourth and sixth terms) and then further adding the first and second terms.

The relation \eqref{iDR3b} can be rewritten as
\begin{align}
  \label{eq:BBBB}
 [\y_{i,l+1}, \y_{i,k}]_{v^{2}}  + [\y_{i,k+1}, \y_{i,l}]_{v^{2}}
&= -(\Theta_{i,l-k+1} C^k -v^{-2}\Theta_{i,l-k-1} C^{k+1}) \K_i
  + \{k \leftrightarrow l \}.
\end{align}
Using \eqref{eq:BBBB},  we rewrite the RHS of the  identity \eqref{eq:Skk2a} as
\begin{align*}
& \big(B_{jl} [B_{i, k_2+1}, B_{i, k_1}]_{v^2} -v^{-2} [B_{i, k_2+1}, B_{i, k_1}]_{v^2} B_{jl} \big)
 + \{k_1 \leftrightarrow k_2\}
 \\
 &=B_{jl} \big( [B_{i, k_2+1}, B_{i, k_1}]_{v^2} + [B_{i, k_1+1}, B_{i, k_2}]_{v^2} \big)
 -v^{-2} \big( [B_{i, k_2+1}, B_{i, k_1}]_{v^2}  + [B_{i, k_1+1}, B_{i, k_2}]_{v^2} \big) B_{jl}
 \\
 &= -B_{jl} \big( \Theta_{i, k_2-k_1+1} C^{k_1} -v^{-2}\Theta_{i,k_2-k_1-1} C^{k_1+1}  \big) \K_i
 \\
 &\quad + v^{-2} \big(\Theta_{i, k_2-k_1+1} C^{k_1} -v^{-2}\Theta_{i,k_2-k_1-1} C^{k_1+1}  \big) B_{jl} \K_i
 + \{k_1 \leftrightarrow k_2\}
\\
&= \Big( -[B_{jl}, \Theta_{i, k_2-k_1+1}]_{v^{-2}} C^{k_1} +v^{-2} [B_{jl}, \Theta_{i, k_2-k_1-1}]_{v^{-2}} C^{k_1+1} \Big) \K_i + \{k_1 \leftrightarrow k_2\}.
\end{align*}
The lemma is proved.
\end{proof}

Denote, for $n\in \Z$,
\begin{align}
  \label{X}
X_{n |l}=&\sum_{p\ge 0} v^{2p+1}[B_{j,l+1},\Theta_{i,n-2p-1}]_{v^{-2}} C^{p+1}
 +\sum_{p\ge 1} v^{2p-1} [2] [\Theta_{i,n-2p},B_{j,l}] C^{p+1}
\\
&\quad -\sum_{p\ge 0} v^{2p+1}[\Theta_{i,n-2p-1},B_{j,l-1}]_{v^{-2}} C^{p+2}
+ [\Theta_{i,n},B_{j,l}] C.
\notag
\end{align}
Clearly, $X_{n |l}=0$, for $n\leq 0$.

\begin{lemma}
\label{lem:vanish1}
If \eqref{iDR2-reform}$_{\le k}$ holds, then  $X_{n |l-1}=0$, for all $n \leq k$ and $l\in \Z$. (The converse is also true.)
\end{lemma}

\begin{proof}



Recalling \eqref{X}, we compute
\begin{align*}
& X_{n |l-1} - v^2 C X_{n-2 |l-1}
\\
%
&= v [B_{j,l},\Theta_{i,n-1}]_{v^{-2}} + v[2][\Theta_{i,n-2},B_{j,l-1}]C - v[\Theta_{i,n-1},B_{j,l-2}]_{v^{-2}}C
\\
&\quad + [\Theta_{i,n},B_{j,l-1}] -v^2[\Theta_{i,n-2},B_{j,l-1}]C
\\
&= -v^{-1}[\Theta_{i,n-1},B_{j,l}]_{v^{2}}+ [\Theta_{i,n-2},B_{j,l-1}]C + [\Theta_{i,n},B_{j,l-1}] -v [\Theta_{i,n-1},B_{j,l-2}]_{v^{-2}}C.
\end{align*}
For $0\le n\le k$, the RHS is $0$, which is equivalent to the assumption that \eqref{iDR2-reform}$_{n}$ holds.
The lemma follows by an induction on $n$ and noting that $X_{-1 |l-1} = X_{0 |l-1}=0$.
\end{proof}

Recall $R(k_1,k_2 |l)$ from \eqref{eq:Rkk}. Let us establish an $R$-counterpart of Lemma~\ref{lem:SSSa}.

\begin{lemma}   \label{lem:RRRa}
For $k_1, k_2, l \in \Z$ with $k_2\ge k_1$, we have
\begin{align*}
& R(k_1,k_2+1 |l) + R(k_1+1,k_2|l) -[2] R(k_1+1,k_2+1|l-1)
 \\
 &= - [2]^2 X_{k_2-k_1 |l-1} C^{k_1} \K_i
+  \big( -[B_{j,l}, \Theta_{i, k_2-k_1+1}]_{v^{-2}} C^{k_1} +v^{-2} [B_{j,l}, \Theta_{i, k_2-k_1-1}]_{v^{-2}} C^{k_1+1} \big) \K_i. 
 \end{align*}
\end{lemma}

\begin{proof}
This proof is based on only formal algebraic manipulations, and does not use any nontrivial relations in $\tUi$.

Following the format of \eqref{eq:Rkk} for $R$, we write
\begin{align*}
R(k_1,k_2+1 |l) &=\big( R_1' +R_2' - [\y_{j,l}, \Theta _{i,k_2-k_1+1}]_{v^{-2}} \big) C^{k_1} \K_i,
\\
R(k_1+1,k_2|l) &= \big( R_1+R_2 - [\y_{j,l}, \Theta _{i,k_2-k_1-1}]_{v^{-2}} C \big) C^{k_1} \K_i,
\\
R(k_1+1,k_2+1 |l-1) &= \big( R_3 +R_4 - [\y_{j,l-1}, \Theta _{i,k_2-k_1}]_{v^{-2}} C \big) C^{k_1} \K_i.
\end{align*}
where $R_1'$ and $R_2'$ denote the first and second summands of $R(k_1,k_2+1 |l)$ as in \eqref{eq:Rkk}; the notations in the other two identities are understood similarly.

By a direct computation, we have
\begin{align*}
R_2' + R_2 &= -\sum_{p\geq0} v^{2p}  [2]^2 [B_{j,l}, \Theta _{i,k_2-k_1-2p-1}]_{v^{-2}} C^{p+1}
+v^{-1}[2]  [B_{j,l}, \Theta _{i,k_2-k_1-1}]_{v^{-2}}C.
\end{align*}

By first summing up $R_1' + R_1$, we also obtain by a direct computation that
\begin{align*}
R_1' + R_1 -[2] R_4
&= -\sum_{p\geq 1} v^{2p-2}  [2]^3 [\Theta _{i,k_2-k_1-2p},\y_{j,l-1}] C^{p+1}
 - [2] [\Theta _{i,k_2-k_1}, B_{j,l-1}]_{v^{-2}} C.
\end{align*}
Note the main summands in $(R_2' + R_2)$, $(R_1' + R_1 -[2] R_4)$ and $-[2] R_3$ above are precisely $- v^{-1}[2]^2$ times the three main summands in $X_{k_2-k_1 |l-1}$ defined in \eqref{X}. Hence we have
\begin{align*}
& R(k_1,k_2+1 |l) + R(k_1+1,k_2|l) -[2] R(k_1+1,k_2+1 |l-1)
\\
 &= (R_2' + R_2) + (R_1' + R_1 -[2] R_4) -[2] R_3
   \notag \\
&\quad -  [B_{j,l}, \Theta _{i,k_2-k_1+1}]_{v^{-2}} - [B_{j,l}, \Theta _{i,k_2-k_1-1}]_{v^{-2}}C
 + [2] [B_{j,l-1}, \Theta_{i,k_2-k_1}]_{v^{-2}} C
  \notag  \\
 & =(- v^{-1} [2]^2 X_{k_2-k_1 |l-1} + v^{-1} [2]^2 [\Theta_{i,k_2-k_1}, B_{j,l-1}] C )
  \notag \\
&\quad +v^{-1}[2]  [B_{j,l}, \Theta _{i,k_2-k_1-1}]_{v^{-2}}C
 - [2] [\Theta _{i,k_2-k_1}, B_{j,l-1}]_{v^{-2}} C
 \notag \\
&\quad -  [B_{j,l}, \Theta _{i,k_2-k_1+1}]_{v^{-2}} - [B_{j,l}, \Theta _{i,k_2-k_1-1}]_{v^{-2}}C
+[2] [B_{j,l-1}, \Theta _{i,k_2-k_1}]_{v^{-2}} C
  \notag \\
 &= - v^{-1} [2]^2 X_{k_2-k_1 |l-1}
 -  [B_{j,l}, \Theta _{i,k_2-k_1+1}]_{v^{-2}} +v^{-2} [B_{j,l}, \Theta _{i,k_2-k_1-1}]_{v^{-2}}C.
 \notag
\end{align*}
The lemma is proved.
\end{proof}

Now we can show that \eqref{iDR2-reform}$_{< k}$ (together with \eqref{eq:S2} and \eqref{iDR3a}--\eqref{iDR3b}) implies \eqref{iDR5-reform}$_{\le k}$, for $k\ge 1$.
\begin{proposition}
 \label{prop:2to5}
Let $k\ge 0$. If \eqref{iDR2-reform}$_{\le k}$ holds, then \eqref{iDR5-reform}$_{\leq k+1}$ holds.
\end{proposition}

\begin{proof}
Assume \eqref{iDR2-reform}$_{\le k}$ holds. Then by Lemma~\ref{lem:vanish1}, $X_{k_2-k_1 |l-1}=0$. Hence by comparing Lemma~\ref{lem:SSSa} and Lemma~\ref{lem:RRRa} we obtain, for $k=k_2-k_1\ge 0$,
\begin{align}
&\SS(k_1,k_2+1 |l) + \SS(k_1+1,k_2|l) -[2] \SS(k_1+1,k_2+1 |l-1)
 \label{eq:SR} \\
=& \R(k_1,k_2+1 |l) + \R(k_1+1,k_2|l) -[2] \R(k_1+1,k_2+1|l-1).
\notag
\end{align}

We induct on $k$. Consider the base case when $k=0$. We have $\SS(r,r |l-1)= \R(r,r |l-1)$, for all $r, l$, by applying $\TT_i^{-r} \TT_j^{1-l}$ to the finite type Serre relation \eqref{eq:S2}. Hence \eqref{iDR5-reform}$_{\le 1}$ holds by \eqref{eq:SR} for $k_1=k_2$ and noting $\SS(k_1,k_1+1 |l) =\SS(k_1+1,k_1 |l)$ and $\R(k_1,k_1+1 |l-1)
=\R(k_1+1,k_1|l-1)$ by the symmetrization definition of $\SS$ and $\R$.

By the inductive assumption, \eqref{iDR5-reform}$_{\le k}$ holds; in particular, we have, for $k=k_2-k_1 >0$,
\[ \SS(k_1+1,k_2|l) = \R(k_1+1,k_2 |l),\qquad \SS(k_1+1,k_2+1 |l-1) =\R(k_1+1,k_2+1|l-1).
\]
We conclude from this and \eqref{eq:SR} that $\SS(k_1,k_2+1 |l) = \R(k_1,k_2+1 |l)$, whence \eqref{iDR5-reform}$_{k+1}$.
\end{proof}

\subsection{Implication from $\eqref{iDR5-reform}_{\le k}$ to  $\eqref{iDR2-reform}_{\leq k}$ }
\label{subsec:iDR5=>iDR2}

We shall fix $i,j \in \II$ such that $c_{ij}=-1$ throughout this subsection.
The strategy here is similar to the strategy used in the previous subsection on the opposite implication. We start with a variant of Lemma~\ref{lem:SSSa}.

\begin{lemma}  \label{lem:SSS}
For $k_1, k_2, l \in \Z$, we have
\begin{align*}
& \SS(k_1,k_2+1 |l) + \SS(k_1+1,k_2|l) -[2] \SS(k_1,k_2|l+1)
 \\
&= \Big( -[\Theta_{i, k_2-k_1+1}, B_{jl}]_{v^{-2}} C^{k_1} +v^{-2} [\Theta_{i, k_2-k_1-1}, B_{jl}]_{v^{-2}} C^{k_1+1} \Big) \K_i + \{k_1 \leftrightarrow k_2\}.
\end{align*}
\end{lemma}
(The proof of the lemma uses only the relations \eqref{iDR3a}--\eqref{iDR3b}.)

\begin{proof}
We rewrite \eqref{eq:Skk} as
  \begin{align*}
\SS(k_1,k_2|l)
&= B_{i,k_1} [B_{i,k_2}, B_{j,l}]_{v} - v^{-1} [B_{i,k_2}, B_{j,l}]_{v} B_{i,k_1}
+\{k_1 \leftrightarrow k_2\}.
\end{align*}
This together with \eqref{iDR3a} implies that
\begin{align}
 \SS(k_1,k_2|l+1)
&=  (B_{i,k_1} [B_{i,k_2}, B_{j,l+1}]_{v} - v^{-1} [B_{i,k_2}, B_{j,l+1}]_{v} B_{i,k_1})
+\{k_1 \leftrightarrow k_2\}
\label{eq:Skk1} \\
&= (v B_{i,k_1} [B_{i,k_2+1}, B_{j,l}]_{v^{-1}} -  [B_{i,k_2+1}, B_{j,l}]_{v^{-1}} B_{i,k_1})
+\{k_1 \leftrightarrow k_2\}.  \notag
\end{align}

Using \eqref{eq:Skk1}, we compute
\begin{align}
  \label{eq:Skk2}
& \SS(k_1,k_2+1 |l) + \SS(k_1+1,k_2|l) -[2] \SS(k_1,k_2|l+1)
\\
&= \Big( \SS(k_1,k_2+1 |l)
-[2] \big(v B_{i,k_1} [B_{i,k_2+1}, B_{j,l}]_{v^{-1}} -  [B_{i,k_2+1}, B_{j,l}]_{v^{-1}} B_{i,k_1} \big) \Big)
+\{k_1 \leftrightarrow k_2\}
 \notag \\
&=
\Big( \big( B_{i,k_1} [B_{i,k_2+1}, B_{j,l}]_{v} - v^{-1} [B_{i,k_2+1}, B_{j,l}]_{v} B_{i,k_1}
+B_{i,k_2+1} [B_{i,k_1}, B_{j,l}]_{v} - v^{-1} [B_{i,k_1}, B_{j,l}]_{v} B_{i,k_2+1} \big)
\notag \\
& \quad -[2] \big( v B_{i,k_1} [B_{i,k_2+1}, B_{j,l}]_{v^{-1}} -  [B_{i,k_2+1}, B_{j,l}]_{v^{-1}} B_{i,k_1} \big) \Big)
 + \{k_1 \leftrightarrow k_2\}
 \notag \\
&= \big( [B_{i, k_2+1}, B_{i, k_1}]_{v^2} B_{jl} -v^{-2} B_{jl} [B_{i, k_2+1}, B_{i, k_1}]_{v^2}  \big)
 + \{k_1 \leftrightarrow k_2\},
 \notag
\end{align}
where the last identity is obtained by first adding the first and fifth terms (and respectively, the second and sixth terms) and then further simplifying when adding with the third and fourth terms.

Using \eqref{eq:BBBB}, we rewrite the RHS of the  identity \eqref{eq:Skk2} as
\begin{align*}
& \big( [B_{i, k_2+1}, B_{i, k_1}]_{v^2} B_{jl} -v^{-2} B_{jl} [B_{i, k_2+1}, B_{i, k_1}]_{v^2} \big)
 + \{k_1 \leftrightarrow k_2\}
 \\
 &=\big( [B_{i, k_2+1}, B_{i, k_1}]_{v^2} + [B_{i, k_1+1}, B_{i, k_2}]_{v^2} \big) B_{jl}
 -v^{-2} B_{jl} \big( [B_{i, k_2+1}, B_{i, k_1}]_{v^2}  +[B_{i, k_1+1}, B_{i, k_2}]_{v^2}  \big)
 \\
 &= -\big( \Theta_{i, k_2-k_1+1} C^{k_1} -v^{-2}\Theta_{i,k_2-k_1-1} C^{k_1+1}  \big) \K_i B_{jl}
 \\
 &\quad + v^{-2} B_{jl} \big(\Theta_{i, k_2-k_1+1} C^{k_1} -v^{-2}\Theta_{i,k_2-k_1-1} C^{k_1+1}  \big) \K_i
 + \{k_1 \leftrightarrow k_2\}
\\
&= \Big( -[\Theta_{i, k_2-k_1+1}, B_{jl}]_{v^{-2}} C^{k_1} +v^{-2} [\Theta_{i, k_2-k_1-1}, B_{jl}]_{v^{-2}} C^{k_1+1} \Big) \K_i + \{k_1 \leftrightarrow k_2\}.
\end{align*}
The lemma is proved.
\end{proof}

Denote
\begin{align}  \label{Y}
Y_{k+1} &=[\y_{j,l+1}, \Theta _{i,k}]_{v^{-2}} - [\Theta _{i,k},\y_{j,l-1}]_{v^{-2}}C
+v^{-1} [\Theta_{i,k+1},B_{j,l}] + v^{-1} [\Theta_{i,k-1},B_{j,l}] C.
\end{align}
(The identity \eqref{iDR2-reform} can be stated equivalently as $Y_{k+1} =0$.)
Here is a variant of Lemma~\ref{lem:RRRa}.

\begin{lemma}   \label{lem:TTT}
For $k_1, k_2, l \in \Z$ with $k_2\ge k_1$, we have
\begin{align*}
& R(k_1,k_2+1 |l) + R(k_1+1,k_2|l) -[2] R(k_1,k_2|l+1)
 \\
 &= [2]^2 X_{k_2-k_1-1 |l} C^{k_1}\K_i + [2] Y_{k_2-k_1+1} C^{k_1} \K_i
 \\
&\quad +  \big( -[\Theta_{i, k_2-k_1+1}, B_{j,l}]_{v^{-2}} C^{k_1} +v^{-2} [\Theta_{i, k_2-k_1-1}, B_{j,l}]_{v^{-2}} C^{k_1+1} \big) \K_i. 
 \end{align*}
\end{lemma}

\begin{proof}
This proof is based on only formal algebraic manipulations, and does not use any nontrivial relations in $\tUi$.

Following the format of \eqref{eq:Rkk} for $R$, we write
\begin{align*}
R(k_1,k_2+1 |l) &=\big( R_1' +R_2' - [\y_{j,l}, \Theta _{i,k_2-k_1+1}]_{v^{-2}} \big) C^{k_1} \K_i,
\\
R(k_1+1,k_2|l) &= \big( R_1+R_2 - [\y_{j,l}, \Theta _{i,k_2-k_1-1}]_{v^{-2}} C \big) C^{k_1} \K_i,
\\
R(k_1,k_2|l+1) &= \big( R_5 +R_6 - [\y_{j,l+1}, \Theta _{i,k_2-k_1}]_{v^{-2}} \big) C^{k_1} \K_i.
\end{align*}
where $R_1'$ and $R_2'$ denote the first and second summands of $R(k_1,k_2+1 |l)$ as in \eqref{eq:Rkk}; the notations in the other two identities are understood similarly.

By a direct computation, we have
\begin{align*}
R_1' + R_1 =-\sum_{p\geq0} v^{2p+1}  [2]^2 [\Theta _{i,k_2-k_1-2p-2},\y_{j,l-1}]_{v^{-2}}C^{p+2}
- [2] [\Theta _{i,k_2-k_1},\y_{j,l-1}]_{v^{-2}}C.
\end{align*}

By first summing up $R_2' + R_2$, we also obtain by a direct computation that
\begin{align*}
R_2' + R_2 -[2] R_5
&= -\sum_{p\geq 1} v^{2p-1}  [2]^3 [\Theta _{i,k_2-k_1-2p-1},\y_{j,l}] C^{p}
\\
&\quad - v[2] [B_{j,l}, \Theta _{i,k_2-k_1-1}]_{v^{-2}}C
+[2]^2 [\Theta _{i,k_2-k_1-1},\y_{j,l}]_{v^{-2}}C
\end{align*}
Note the main summands in $(R_1' + R_1)$, $(R_2' + R_2 -[2] R_5)$ and $-[2] R_6$ above are precisely $[2]^2$ times the three main summands (in reversed order) in $X_{k_2-k_1}$ defined in \eqref{X}. Hence we have
\begin{align}
& R(k_1,k_2+1 |l) + R(k_1+1,k_2|l) -[2] R(k_1,k_2|l+1)
 \label{eq:TTT2} \\
 &= (R_1' + R_1) + (R_2' + R_2 -[2] R_5) -[2] R_6
   \notag \\
&\quad - [\y_{j,l}, \Theta _{i,k_2-k_1+1}]_{v^{-2}}
 - [\y_{j,l}, \Theta _{i,k_2-k_1-1}]_{v^{-2}} C
 +[2] [\y_{j,l+1}, \Theta _{i,k_2-k_1}]_{v^{-2}}
 \notag \\
 &= [2]^2 (X_{k_2-k_1-1 |l} - [\Theta_{i,k_2-k_1-1},B_{j,l}] C)
 - [2] [\Theta _{i,k_2-k_1},\y_{j,l-1}]_{v^{-2}}C
 \notag \\
 &\quad - v[2] [B_{j,l}, \Theta _{i,k_2-k_1-1}]_{v^{-2}}C
+[2]^2 [\Theta _{i,k_2-k_1-1},\y_{j,l}]_{v^{-2}}C
 \notag \\
&\quad - [\y_{j,l}, \Theta _{i,k_2-k_1+1}]_{v^{-2}}
 - [\y_{j,l}, \Theta _{i,k_2-k_1-1}]_{v^{-2}} C
 +[2] [\y_{j,l+1}, \Theta _{i,k_2-k_1}]_{v^{-2}}.
  \notag
\end{align}
On the RHS of \eqref{eq:TTT2} above, there is exactly one term involving $B _{i,l+1}$ and one term involving $B_{i,l-1}$, with opposite coefficients, just as those appearing in $Y_{k_2-k_1+1}$ \eqref{Y}.  This allows us to rewrite the RHS of \eqref{eq:TTT2} in terms of $X_{k_2-k_1-1 |l}, Y_{k_2-k_1+1}$ and the terms involving $B_{j,l}$ only. The terms involving $B_{j,l}$ can then be simplified by some direct computation (without using any relations) to the formula stated in the lemma.
\end{proof}

Now we can  show that \eqref{iDR5-reform}$_{\le k}$ (together with \eqref{eq:S2} and \eqref{iDR3a}--\eqref{iDR3b}) implies \eqref{iDR2-reform}$_{\le k}$.
\begin{proposition}
 \label{prop:5to2}
Let $k\ge 0$. If \eqref{iDR5-reform}$_{\le k}$ holds, then \eqref{iDR2-reform}$_{\le k}$ holds.
\end{proposition}

\begin{proof}
We prove by induction on $k$. The identity \eqref{iDR2-reform}$_{k}$ with $k=0$ is trivial.  
The case \eqref{iDR2-reform}$_{k+1}$ (i.e., $Y_{k+1}=0$) follows by comparing the identities in Lemma~\ref{lem:SSS}
and Lemma~\ref{lem:TTT} (where $k=k_2 -k_1 \ge 0$), with the help of Lemma~\ref{lem:vanish1} and
the inductive assumption \eqref{iDR5-reform}$_{\le k+1}$. (See the proof of Proposition~\ref{prop:2to5} for more details of the same type of arguments.)
\end{proof}

%

\begin{proposition}
 \label{prop:iDR25}
The relations \eqref{iDR2} and \eqref{iDR5} for $c_{ij}=-1$  hold in $\tUi$.
\end{proposition}

\begin{proof}
These relations follow by the equivalent identities \eqref{iDR2-reform}--\eqref{iDR5-reform}, which have been established inductively and simultaneously in Proposition~\ref{prop:2to5} and Proposition~ \ref{prop:5to2}.
\end{proof}

\section{Variants of Drinfeld type presentations}
  \label{sec:variants}

In this section, we formulate several variants of the Drinfeld type presentation for $\tUi$, one in generating function formalism, one in a more symmetrized form, and another via different imaginary root vectors. We also deduce a Drinfeld type presentation for the $\imath$quantum group $\Ui_\bvs$ with parameters $\bvs =(\vs_i)_{i\in \I}$.

\subsection{Presentation via generating functions}

Recall the generating functions $\Y_{i}(z), \bTH_i(z)$ and $\bDel(z)$ from \eqref{eq:Genfun}.

\begin{theorem}
  \label{thm:ADEgf}
$\tUiD$ is generated by $\K_{i}^{\pm1}$, $C^{\pm1}$, $\Theta_{i,k}$ and $\y_{i,l}$ $(i\in \II$, $k\geq1$, $l\in\Z)$, subject to the following  relations, for $i, j \in \II$:
\begin{align}
&\K_i \text{ are central, }\quad  \bTH_i(z) \bTH_j(w) =\bTH_j(w) \bTH_i(z),
\label{iDRG1}
\\
& \bTH_i (z)  \bB_j(w)
=   \frac{(1 -v^{-c_{ij}}zw^{-1}) (1 -v^{c_{ij}} zw C)}{(1 -v^{c_{ij}}zw^{-1})(1 -v^{-c_{ij}}zw C)}
  \bB_j(w) \bTH_i (z),
 \label{iDRG2}
\\
&\Y_i(w)\Y_j(z)=\Y_j(z)\Y_i(w), \qquad\qquad\qquad\text{ if }c_{ij}=0,  \label{iDRG4}
\\
&(v^{c_{ij}}z -w) \Y_i(z) \Y_j(w) +(v^{c_{ij}}w-z) \Y_j(w) \Y_i(z)=0, \qquad \text{ if }i\neq j, \label{iDRG3a}
\\
&(v^2z-w) \Y_i(z) \Y_i(w) +(v^{2}w-z) \Y_i(w) \Y_i(z)   \label{iDRG3b}
\\\notag
&\quad =\frac{v^{-2}}{v-v^{-1}} \K_{i} \bDel(zw) \big( (v^2z-w)\bTH_i(w) +(v^2w-z)\bTH_i(z) \big),
\\
&\Y_i(w_1)\Y_i(w_2)\Y_{j}(z) -[2]\Y_i(w_1)\Y_{j}(z)\Y_i(w_2)+\Y_{j}(z)\Y_i(w_1)\Y_i(w_2)+ \{w_1\leftrightarrow w_2\}
 \label{iDRG5} \\
=&  -\K_{i} \frac{\bDel(w_1w_2)}{v-v^{-1}} \Big( [\bTH_i(w_2),\bB_j(z)]_{v^{-2}} \frac{[2] z w_1^{-1} }{1 -v^{2}w_2w_1^{-1}}
+  [\bB_j(z),\bTH_i(w_2)]_{v^{-2}}  \frac{1 +w_2w_1^{-1}}{1 -v^{2}w_2w_1^{-1}} \Big)
\notag \\
& \qquad  + \{w_1\leftrightarrow w_2\},
\qquad\qquad
\text{ if }c_{ij}=-1.
\notag
\end{align}
\end{theorem}

\begin{proof}
We simply rewrite the relations \eqref{iDR1}--\eqref{iDR5} in Theorem \ref{def:iDR} by using the generating functions \eqref{eq:Genfun}.
The relation \eqref{iDRG1} is clear. As formulated in Proposition~\ref{prop:equivij}, the relation \eqref{iDRG2} is equivalent to \eqref{iDR2}.
The relation \eqref{iDRG3b} is obtained from \eqref{iDR3b} by multiplying both side of the relation \eqref{iDR3b} by $v^2z^{r+1} w^{s+1}$ and summing over $r,s\in\Z$.
Similarly, the relations \eqref{iDRG4} and \eqref{iDRG3a} are equivalent to the relations \eqref{iDR4} and \eqref{iDR3a}, respectively.
Finally, the relation \eqref{iDRG5} is obtained by multiplying both sides of the relation \eqref{iDR5} by $w_1^{k_1}w_2^{k_2}z^{r}$ and summing over $k_1,k_2,r\in\Z$.
\end{proof}

\subsection{A symmetrized presentation}

We add two central generators $C^{\pm\frac{1}{2}}$ such that $C^{\frac{1}{2}} C^{-\frac{1}{2}}=1=C^{-\frac{1}{2}}C^{\frac{1}{2}}$ and $(C^{\frac{1}{2}})^2=C$.

\begin{definition}
\label{def:i-DR-ref}
Let ${}^{\text{DR}}\tUi$ be the $\Q(v)$-algebra generated by $\K_{i}^{\pm1}$, $C^{\pm\frac{1}{2}}$, $\tH_{i,m}$ and $\y_{i,l}$, where  $i\in \II$, $m\geq1$, $l\in\Z$, subject to the following  relations, for $m,n\ge 1$ and $k,l\in \Z$:
\begin{align}
&\K_i, C^{\pm\frac{1}{2}} \text{ are central, } \quad  [\tH_{i,m},\tH_{j,n}]=0, \label{iDRC1}
\\
&[\tH_{i,m},\y_{jl}]=\frac{[mc_{ij}]}{m} \y_{j,l+m}C^{-\frac{m}{2}}-\frac{[mc_{ij}]}{m} \y_{j,l-m}C^{\frac{m}{2}}, \label{iDRC2}
\\
&[\y_{i,k} ,\y_{j,l}]=0, \quad \text{ if }c_{ij}=0, \label{iDRC4}
\\
\label{iDRC3a}
&[\y_{i,k},\y_{j,l+1}]_{v^{-c_{ij}}}-v^{-c_{ij}}[\y_{i,k+1},\y_{j,l}]_{v^{c_{ij}}}=0, \text{ if }i\neq j,
\\
&[\y_{i,k},\y_{i,l+1}]_{v^{-2}}-v^{-2}[\y_{i,k+1},\y_{i,l}]_{v^{2}} \label{iDRC3}
\\\notag
&= \K_i  C^{\frac{k+l+1}{2}}
\Big( v^{-2}\tTH_{i,l-k+1} -v^{-4} \tTH_{i,l-k-1} +v^{-2} \tTH_{i,k-l+1} -v^{-4} \tTH_{i,k-l-1}\Big),
\\
& B_{i,k_1} B_{i,k_2} B_{j,l} -[2] B_{i,k_1} B_{j,l} B_{i,k_2} + B_{j,l} B_{i,k_1} B_{i,k_2} +\{k_1 \leftrightarrow k_2\}
  \label{iDRC5} \\
&= \K_i C^{ \frac{k_1+k_2}{2}}
\Big(-\sum_{p\geq0}v^{2p}[2] [\tTH_{i,k_2-k_1-2p-1},\y_{j,l-1}]_{v^{-2}} C^{\frac{1}{2}}\notag
\\\notag
&-\sum_{p\geq0}v^{2p-1}[2][\y_{j,l},\tTH_{i,k_2-k_1-2p}]_{v^{-2}}
+v^2[\y_{j,l}, \tTH_{i,k_2-k_1}]_{v^{-2}}  \Big) +\{k_1 \leftrightarrow k_2\},
\text{ if }c_{ij}=-1.
\notag
\end{align}
Here 
 $\tTH_{i,m}$, for $i \in \II$ and $m\ge 1$,  are defined by the following equation:
\begin{align}
\label{exp hC}
1+ \sum_{m\geq 1} &(v-v^{-1})\tTH_{i,m} u^m  = \exp\big( (v-v^{-1}) \sum_{m=1}^\infty \tH_{i,m} u^m \big);
\end{align}
we also set $\tTH_{i,0}=\frac{1}{v-v^{-1}}$, $\tTH_{i,m}=0$ for $m<0$.
\end{definition}

We enlarge $\tUiD$ to a $\Q(v)$-algebra $\tUiD [C^{\pm\frac12}] := {}^{\text{Dr}}\tUi \otimes_{\Q(v)[C^{\pm 1}]} \Q(v)[C^{\pm \frac12}]$; recall it contains the generators $H_m$.
For $i \in \II$ and $m\ge 1$, we identify
\begin{align}  \label{eq:HH2}
\tH_{i,m}=H_{i,m} C^{-\frac{m}{2}},\qquad \tTH_{i,m}=\Theta_{i,m} C^{-\frac{m}{2}}.
\end{align}
Then \eqref{exp hC} holds.
Now such an identification \eqref{eq:HH2} leads to an isomorphism ${}^{\text{DR}}\tUi \cong {}^{\text{Dr}}\tUi [C^{\pm\frac12}]$ in the proposition below, which also identifies elements in the same notation. We skip the detail as the verification is straightforward.
\begin{proposition}
  \label{prop:symm}
We have a natural $\Q(v)$-algebra isomorphism
${}^{\text{DR}}\tUi \cong {}^{\text{Dr}}\tUi [C^{\pm\frac12}]$.
\end{proposition}

\subsection{Presentation via different root vectors}

Recall starting from the rank 1 case treated in Section~\ref{sec:Onsager}, we have preferred the imaginary root vectors $\Theta_m$ over $\acute{\Theta}_m$, because of a consideration from Hall algebra \cite{LRW20}. Choosing $\acute{\Theta}_m$ will lead to the following presentation for $\tUi$.

\begin{theorem}
  \label{thm:ADE1}
The algebra $\tUi$ admits a presentation with the same set of generators of ${}^{\text{Dr}}\tUi$ as in Definition~\ref{def:iDR}, subject to the relations~\eqref{iDR1}--\eqref{iDR3a} for ${}^{\text{Dr}}\tUi$ and the following 2 relations \eqref{DrBB}--\eqref{DrSerre} (in place of \eqref{iDR3b}--\eqref{iDR5} for ${}^{\text{Dr}}\tUi$):
\begin{align}
&[\y_{i,r}, \y_{i,s+1}]_{v^{-2}}  -v^{-2} [\y_{i,r}, \y_{i,s}]_{v^{2}}
   \label{DrBB}
=v^{-2} {\Theta}_{i,s-r+1} C^r \K_i-v^{-2} {\Theta}_{i,s-r-1} C^{r+1} \K_i \\
 \notag
&+v^{-2} {\Theta}_{i,r-s+1} C^s \K_i-v^{-2} {\Theta}_{i,r-s-1} C^{s+1} \K_i,
\\
& B_{i,k_1} B_{i,k_2} B_{j,l} -[2] B_{i,k_1} B_{j,l} B_{i,k_2} + B_{j,l} B_{i,k_1} B_{i,k_2} +\{k_1 \leftrightarrow k_2\}
 \label{DrSerre}
\\
&=
- \textstyle  C^{k_1} \K_i \Big(\sum_{p\geq0} (v^{-2p-1} +v^{2p+1}) [ {\Theta}_{i,k_2-k_1-2p-1},\y_{j,l-1}]_{v^{-2}}C^{p+1} \notag
\\\notag
& \textstyle
\quad +\sum_{p\geq 1} (v^{-2p} +v^{2p} )  [\y_{j,l}, {\Theta} _{i,k_2-k_1-2p}]_{v^{-2}} C^{p}
+[\y_{j,l}, {\Theta} _{i,k_2-k_1}]_{v^{-2}}    \Big)
  \notag \\
 &\quad +\{k_1 \leftrightarrow k_2\},
\qquad \text{ if }c_{ij}=-1.
\notag
\end{align}
\end{theorem}

\begin{proof}
In this proof, we shall denote the generators ${\Theta}_{i,k}$ in the proposition by $\acute{\Theta}_{i,k}$, in order to distinguish the ${\Theta}_{i,k}$ used for $\tUiD$ in Definition~\ref{def:iDR}.

By Theorem~\ref{thm:ADE}, it suffices to show the algebra with presentation given by the proposition is isomorphic to $\tUiD$. The isomorphism is given by matching generators in the same notation and in addition imposing the relation between ${\Theta}_{i,k}$ and $\acute{\Theta}_{i,k}$ as follows (cf. \eqref{eq:Theta2}):
\begin{align*}
\bTH_i(z) =  \frac{1- Cz^2}{1-v^{-2}C z^2}  \acute{\bTH}_i(z).
\end{align*}
Then the equivalence between the relations \eqref{iDR3b} and \eqref{DrBB} follows from the equivalence between \eqref{rel:iDr1} and \eqref{rel:iDr1b}. Finally, the equivalence between the relations \eqref{iDR5} and \eqref{DrSerre} follows by a direct computation, which we omit here.
\end{proof}

\subsection{Presentation of affine $\imath$quantum groups with parameters}
  \label{subsec:parameter}

Fix $\bvs =(\vs_i)_{i\in \I}$, where $\vs_i \in \Q(v)^\times$, for each $i$. A quantum symmetric pair $(\U, \Ui_\bvs)$ was first formulated by G. Letzter in finite type and then generalized by Kolb \cite{Ko14}.
 The $\imath$quantum group $\Ui_\bvs$ is the $\Q(v)$-algebra with generators $B_i (i\in \I)$, and it can be obtained from $\tUi$ by a central reduction \cite{LW19a} (recall $\tilde{k}_i$ was used {\em loc. cit.} and it is related to $\K_i$ in this paper by $\K_i =-v^2\tilde{k}_i$):
\begin{align}  \label{Ui}
 \Ui_\bvs =\tUi / (\K_i +v^2 \vs_i \mid i \in \I).
 \end{align}
For $\de=\sum_{i\in \I}a_i\alpha_i$, we define
\[
\delta_\bvs=\prod_{i\in \I} (-v^2\vs_i)^{a_i} \in \Q(v).
\]

From \eqref{Ui} we obtain a natural surjective homomorphism of $\Q(v)$-algebra
 \[
 \pi: \tUi \longrightarrow \Ui_\bvs, \qquad B_i \mapsto B_i \; (i\in \I).
 \]
 By abuse of notations, we shall keep using the same notations for the images under $\pi$ of various elements such as $B_{i,k}, \Theta_{i,m}, H_{i,m}$, for $i\in \II, k\in \Z, m\ge 1$. The algebra $\Ui_\bvs$ has a Serre-type presentation with generators $B_i$, for $i\in \I$, and defining relations \eqref{eq:S1}--\eqref{eq:S3}, where $\K_i$ is replaced by $-v^2 \vs_i$, cf. \cite{BB10, Ko14}. Below we present a Drinfeld type presentation for $\Ui_\bvs$.

\begin{theorem}  \label{thm:ADE2}
Let $\bvs=(\vs_i)_{i\in \I}\in \Q(v)^{\times, \I}$. Then the $\Q(v)$-algebra $\Ui_\bvs$ has a presentation with generators
$H_{i,m}$ and $\y_{i,l}$, where  $i\in \II$, $m\geq1$, $l\in\Z$ and the following relations, for $r,s \in \Z$, $m,n \ge 1$, $i,j\in \II$:
\begin{align}
&[H_{i,m},H_{j,n}]=0,  
\\
&[H_{i,m},\y_{j,l}]=\frac{[mc_{ij}]}{m} \y_{j,l+m}-\frac{[mc_{ij}]}{m} \delta_{\bvs}^m \y_{j,l-m},
\\
&[\y_{i,k} ,\y_{j,l}]=0,  \text{ if }c_{ij}=0, 
\\
&[\y_{i,k}, \y_{j,l+1}]_{v^{-c_{ij}}}  -v^{-c_{ij}} [\y_{i,k}, \y_{j,l}]_{v^{c_{ij}}}=0, \text{ if }i\neq j,
\\
&[\y_{i,k}, \y_{i,l+1}]_{v^{-2}}  -v^{-2} [\y_{i,k}, \y_{i,l}]_{v^{2}}
=- \delta_{\bvs}^k\vs_i \Theta_{i,l-k+1}  +v^{-2} \delta_{\bvs}^{k+1} \vs_i \Theta_{i,l-k-1}   \\\notag
&\qquad\qquad\quad\qquad\qquad\qquad\qquad\qquad
- \delta_{\bvs}^l \vs_i \Theta_{i,k-l+1}  +v^{-2} \delta_{\bvs}^{l+1} \vs_i\Theta_{i,k-l-1},
\notag
\\
& B_{i,k_1} B_{i,k_2} B_{j,l} -[2] B_{i,k_1} B_{j,l} B_{i,k_2} + B_{j,l} B_{i,k_1} B_{i,k_2} +\{k_1 \leftrightarrow k_2\}
\\
=& v^2 \vs_i \delta_{\bvs}^{k_1}
\left(\sum_{p\geq0} v^{2p}[2] \delta_{\bvs}^{p+1} [\Theta _{i,k_2-k_1-2p-1},\y_{j,l-1}]_{v^{-2}} \right. \notag
\\
& \left. +\sum_{p\geq 1} v^{2p-1}[2] \delta_{\bvs}^{p} [\y_{j,l},\Theta _{i,k_2-k_1-2p}]_{v^{-2}}
+ [\y_{j,l}, \Theta _{i,k_2-k_1}]_{v^{-2}}     \right)
+\{k_1 \leftrightarrow k_2\},
\text{ if }c_{ij}=-1.
\notag
\end{align}
\end{theorem}
The theorem above, for which we continue to adopt the convention \eqref{Hm0}--\eqref{exp h}, follows directly from Theorem~\ref{thm:ADE} and \eqref{Ui}.


\end{document}